\documentclass[10pt,leqno]{amsart}

\usepackage[margin=1in]{geometry}

\usepackage{amsmath,amsfonts,amssymb,amscd,amsthm,amsbsy,upref}
\usepackage{esint}
\usepackage{xcolor}
\usepackage[english]{babel} 
\usepackage{mathrsfs}

\numberwithin{equation}{section}
\newtheorem{thm}{Theorem}[section]
\newtheorem{lem}[thm]{Lemma}
\newtheorem{cor}[thm]{Corollary}
\newtheorem{prop}[thm]{Proposition}

\theoremstyle{definition}
\newtheorem{defn}[thm]{Definition}

\def\A{{\mathcal A}}

\def\L{{\mathcal L}}

\def\e{\varepsilon}
\def\A{{\mathcal A}}
\def\L{{\mathcal L}}

\def\R{{\mathbb R}}
\def\N{{\mathbb N}}
\def\H{{\mathcal H}}
\def\Om{{\Omega}}
\def\p{{\partial}}
\def\na{{\nabla}}

\newcommand\dist{\text{dist}}
\newcommand\I[1]{\chi_{\{#1>0\}}}  
\newcommand\fb[1]{\p \{ {#1} > 0 \}}

\newcommand\fbr[1]{\p_{\rm red} \{ {#1} > 0 \}}

\renewcommand{\leq}{\leqslant}
\renewcommand{\le}{\leqslant}
\renewcommand{\geq}{\geqslant}
\renewcommand{\ge}{\geqslant}
\renewcommand{\iota}{\Upsilon}

\renewcommand{\epsilon}{\varepsilon}

\begin{document}
\title[A nonlinear free boundary problem with 
a self-driven Bernoulli condition]{A nonlinear free boundary problem \\
with a self-driven Bernoulli condition}

\author{Serena Dipierro}
\address{School of Mathematics and Statistics,
University of Melbourne,
Richard Berry Building,
Parkville VIC 3010,
Australia.}
\email{s.dipierro@unimelb.edu.au}

\author{Aram Karakhanyan}
\address{Maxwell Institute for
Mathematical Sciences and School of Mathematics, University of
Edinburgh, James Clerk Maxwell Building, Peter Guthrie Tait Road,
Edinburgh EH9 3FD, United Kingdom.}
\email{aram.karakhanyan@ed.ac.uk}

\author{Enrico Valdinoci}
\address{School of Mathematics and Statistics,
University of Melbourne,
Richard Berry Building,
Parkville VIC 3010,
Australia, 
and Weierstra{\ss} Institut
f\"ur Angewandte Analysis und
Stochastik, Mohrenstra{\ss}e 39, 10117 Berlin, Germany,
and Dipartimento di Matematica Federigo Enriques,
Universit\`a degli Studi di Milano,
Via Cesare Saldini 50, 20133 Milano, Italy,
and Istituto di Matematica Applicata 
e Tecnologie Informatiche,
Consiglio Nazionale delle Ricerche,
Via Ferrata 1, 27100 Pavia, Italy.}
\email{enrico@math.utexas.edu}


\maketitle

\baselineskip=14pt              

\begin{abstract}
We study a Bernoulli type free boundary problem with two phases
$$J[u]=\int_{\Om}|\na u(x)|^2\,dx+\Phi\big({\mathcal M}_-(u), 
{\mathcal{M}}_+(u)\big), \quad u-\bar u\in W^{1,2}_0(\Omega),
$$
where $\bar u\in W^{1,2}(\Om)$ is a given boundary datum.
Here, ${\mathcal M}_1$ and ${\mathcal M}_2$ are weighted volumes 
of $\{u\le 0\}\cap \Om$ and $\{u>0\}\cap \Om$, respectively, 
and $\Phi$ is a nonnegative function of two real variables.

We show that, for this problem, the Bernoulli constant, which determines 
the gradient jump condition across the free boundary, 
is of global type and it is indeed determined by the weighted 
volumes of the phases.

In particular, the Bernoulli condition that we obtain can be
seen as a pressure prescription in terms of the volume of the two
phases of the minimizer
itself (and therefore it depends on the minimizer itself and not only on
the structural constants of the problem).

Another property of this type
of problems is that the minimizer in $\Om$ is not necessarily 
a minimizer in a smaller subdomain, due to the nonlinear structure of the problem. 

Due to these features, this problem is highly unstable as 
opposed to the classical case studied by Alt, Caffarelli and Friedman. 
It also interpolates the classical case, in the sense that 
the blow-up limits of $u$ are minimizers of the Alt-Caffarelli-Friedman functional.
Namely, the energy of the problem somehow
linearizes in the blow-up limit.  

As a special case, we can deal with
the energy levels generated by the volume 
term $\Phi(0, r_2)=r_2^{\frac{n-1}n}$, which interpolates 
the Athanasopoulos-Caffarelli-Kenig-Salsa energy, 
thanks to the isoperimetric inequality. 

In particular, we develop a detailed optimal regularity theory for the minimizers
and for their free boundaries.
\end{abstract}

\setcounter{tocdepth}{1}
\tableofcontents

\section{Introduction}

After~\cite{AC, ACF}, a classical problem in the free boundary theory 
consists in studying the minimizers of an energy functional
which is the {\em linear superposition}
of a Dirichlet energy and a volume term.
In this case, minimizers are proved to be harmonic away from the free boundary.
Also, minimizers naturally enjoy a free boundary condition
which can be seen as a balance of the normal derivatives across the interface.

This type of problems
has a natural interpretation in terms
of two dimensional flows of two irrotational, incompressible
and inviscid fluids. Indeed, if the fluids have velocities 
${\bf v^\pm}=\nabla \phi^\pm$,
for some potential functions~$\phi^\pm$,
it holds that $\Delta \phi^\pm=0$ whenever $\phi^\pm\not =0$. 
In addition, the  Bernoulli law  states that  
\begin{equation}\label{CLASSICAL}
\frac{p^\pm(x)}{\rho^\pm}=\frac{|\bf v^\pm|}2+C_\pm\end{equation}
along every streamline, i.e. lines for which the tangent is in the direction of the velocity 
(in other words the level sets of the harmonic conjugate $\psi^\pm$ of $\phi^\pm$).
Here~$C_\pm$ are constants depending on the streamline and $p^\pm$ is the pressure from either side.
If the free boundary is smooth then the pressure $p^\pm$ is continuous and therefore from Cauchy-Riemann equations, after normalization, we  get 
that  (assuming that the densities $\rho^\pm$ are constant)
\begin{equation}\label{BERc}
p^\pm(x)+\frac{|\nabla \psi^\pm|^2}2=C_\pm  \quad 
\Longrightarrow \quad |\nabla \psi^+|^2-|\nabla \psi^-|^2=2(C_+-C_-).\end{equation}
In this interpretation, we see that the free boundary condition in~\cite{AC, ACF}
is a variational version of the classical Bernoulli law (and in fact
it justifies the validity of a weak version of this law at points where
the free boundary is not regular).\medskip

In this paper, we consider the case in which the energy functional is
a {\em nonlinear superposition} 
of a Dirichlet energy and a volume term.

We will show that general nonlinearities may produce pathologic examples,
in which minimizers may not exist, or in which the free boundary of
the minimizers is not smooth. Nevertheless, under suitable structural
assumptions on the nonlinearity, we will show that a sufficiently strong
existence and regularity theory holds true. 

In addition, we will obtain a new version of the free boundary condition,
which, in our case, turns out to be of ``global'' type.
As a matter of fact, in our case, the free boundary condition
may still be seen as a balance between
the normal derivatives
from the two sides of the free boundary, but, differently from the classical
case, this balance changes from point to point of the free boundary
and the change takes into account quantities that are defined globally,
and not only locally (e.g., they include the nonlinearity itself and
the weighted volumes of the phases
of the minimizers).

Roughly speaking, in this new free boundary condition,
the quantities~$C_\pm$ in~\eqref{CLASSICAL}
are not constant anymore and they are not locally determined.
In other words, they 
depend not only on the streamline but also 
on the weighted volumes that the streamline separates,
and, above all, on the minimizers themselves:
for this reason, we named this type of condition {\em{self-driven}}.
\medskip

An explicit geometric example related to
our problem can be given in terms of the isoperimetric inequality
\[\frac{Area(\partial \Omega^+)}{\left[\mathcal L^n(\Omega^+)\right]^{1-\frac1n}}\ge 
\frac{Area(\mathbb S^{n-1})}{
\left[\mathcal L^n(B_1)\right]^{1-\frac1n}}=: c_n,\]
that is
\[{Area(\partial \Omega^+)}\ge c_n 
{\left[\mathcal L^n(\Omega^+)\right]^{1-\frac1n}}.\]
Consequently, if~$\Omega^+:=\{u>0\}$ and~$\Phi_0(r)\le Cr^{1-\frac1n}$
for some~$C>0$, we have that
\begin{equation}\label{drtyfshsjkfgva1}\begin{split}&
J_{\rm{ACKS}}[u]:=\int_\Omega|\nabla u|^2+{Area(\partial\Omega^+)}\ge 
\int_\Omega|\nabla u|^2+c_n \left[\mathcal L^n(\Omega^+)\right]^{1-\frac1n}
\\ &\qquad\quad\ge \int_\Omega|\nabla u|^2+c_o\,\Phi_0
\left(\mathcal L^n(\Omega^+)\right)
=:J[u],\end{split}\end{equation}
with~$c_o:=c_n/C$.
In this sense, the energy functional~$J_{\rm{ACKS}}$
studied in~\cite{ACKS} provides an upper bound for the energy functional~$J$.
Notice that~$J_{\rm{ACKS}}$ is a {\em linear interpolation}
of energies (the second one being an area),
while~$J$ is a {\em nonlinear interpolation} of energies
(the second one being of volume type, but scaling like an area).
The functional~$J$ in~\eqref{drtyfshsjkfgva1} is indeed a model
case for the ones that we study in the present paper (see below for precise assumptions).

We observe that this type of problems is related to the 
Ginzburg-Landau model with three competing rates 
which balance each other for a suitable choice of the structural
parameter. The exact choice of the rate gives, 
in the limit, the energy $J_{\rm ACKS}$.
Thus, in this spirit, the functional~$J$ 
in~\eqref{drtyfshsjkfgva1}
describes a model in which the equilibrium is reached 
in terms of the
best approximation of isoperimetric inequality under given constraints.

In the following subsections, we will describe the formal mathematical
setting of the problem, the main results and the organization
of this paper.

\subsection{Problem set-up}
Let~$\Om\subset\R^n$ be an open and bounded set. In what follows,
$\lambda_1\ge0$ and~$\lambda_2>0$ are given constants.
For a given measurable function  $Q:\Om\to\R$, bounded by two positive numbers    
\begin{equation}\label{Q}
0<Q_1\leq Q(x)\le Q_2<\infty,
\end{equation}
we define the weighted partial volumes 
\begin{equation*}
{\mathcal{M}}_1(u):= 
\lambda_1\int_\Omega Q(x)\,\chi_{\{ u\le0 \}}(x)\,dx\quad{\mbox{ and }}\quad
{\mathcal{M}}_2(u):=
\lambda_2\int_\Omega Q(x)\,\chi_{\{ u>0 \}}(x)\,dx 
\end{equation*}
and the total volume 
$$ \lambda_\Omega:= \int_\Omega Q(x)\,dx.$$
It is easy to see that
\begin{equation}\label{UNASO}
{\mathcal{M}}_1(u)= 
\lambda_1\left(\lambda_\Omega -
\int_\Omega Q(x)\,\chi_{\{ u>0 \}}(x)\,dx \right)
=\lambda_1\,\big(\lambda_\Omega-\lambda_{2}^{-1}{\mathcal{M}}_2(u)\big).
\end{equation}
Let also~$\R_+:=\{x\in \R: x>0\}$. For a 
given ~$\Phi\in C^0\big( \overline{\R_+}\times\overline{\R_+},\,
\overline{\R_+}\big)\cap
C^1 \big( \R_+\times \R_+,\, \R_+\big)$
such that ~$\Phi(0,0)=0$, we consider 
\begin{equation}\label{PHI-s-12}
\Phi_0(r):= \Phi\Big( \lambda_1\,\big(\lambda_\Omega-
\lambda_{2}^{-1}r\big),\,r\Big)\end{equation}
and suppose that
\begin{equation}\label{LAMBDA12}
{\mbox{$\Phi_0'(r)\ge0$
for any~$r\in(0,\,\lambda_2\,\lambda_\Omega)$.}}\end{equation}
In view of~\eqref{UNASO},
\begin{equation*}
\Phi\big({\mathcal{M}}_1(u),{\mathcal{M}}_2(u)\big) =\Phi_0
\big({\mathcal{M}}_2(u)\big).
\end{equation*}
In this paper we study the minimization problem of the energy functional
\begin{equation}\label{EN:FUNCT}\begin{split}
J[u]\,&:=
\int_{\Om}|\na u(x)|^2\,dx+ \Phi\big({\mathcal{M}}_1(u),{\mathcal{M}}_2(u)\big)\\
&=\int_{\Om}|\na u(x)|^2\,dx+\Phi_0\big({\mathcal{M}}_2(u)\big)\end{split}
\end{equation}
in the admissible class
\begin{equation}\label{class}
\A:=\{u\in W^{1, 2}(\Om), {\mbox{ with }}
u-\bar u \in W^{1, 2}_0(\Om)\}, \end{equation}
where~$\bar u\in W^{1, 2}(\Om)$.

For a given minimizer $u$ of ~\eqref{EN:FUNCT}  the free
boundary of~$u$ is denoted by~$\Gamma:= \partial \Omega^+(u)$,
where
\begin{equation}\label{omegapiu}
\Omega^+(u) := \{ x\in\Omega {\mbox{ s.t. }} u(x)>0\}.
\end{equation}
This problem
can be viewed as an extrapolation of the classical free boundary problem
of Alt and Caffarelli~\cite{AC}, where  the authors studied the local minimizers of the energy
\begin{equation}\label{EN:FUNCT:AC}
J_{\text{AC}}[u]:=\int_{\Om} \Big( |\na u(x)|^2+Q(x)\,\chi_{\{u>0\}}(x)\Big)\,dx.
\end{equation}
Indeed, the functional in~\eqref{EN:FUNCT}
reduces to that in~\eqref{EN:FUNCT:AC} with the choices~$\lambda_1:=0$,
$\lambda_2:=1$ and~$\Phi(r_1,r_2):=r_2$.

More generally, the functional in~\eqref{EN:FUNCT} is also
an extrapolation of the two-phase free boundary problem
in~\cite{ACF}, in which, instead of the functional in~\eqref{EN:FUNCT},
the minimization problem dealt with
the energy
\begin{equation}\label{EN:FUNCT:ACF}
J_{\text{ACF}}[u]:=\int_{\Om} \Big( |\na u(x)|^2+\lambda_1\,
Q(x)\,\chi_{\{u\le0\}}(x) +\lambda_2\,Q(x)\,\chi_{\{u>0\}}(x)\Big)\,dx,
\end{equation}
since the functional in~\eqref{EN:FUNCT}
reduces to that in~\eqref{EN:FUNCT:ACF} with the choice~$\Phi(r_1,r_2):=r_1+r_2$
(in this case, condition~\eqref{LAMBDA12} reduces to~$\lambda_2\ge\lambda_1$,
compare with the assumptions of Theorem~2.3 in~\cite{ACF}).

In this sense, the energy functional in~\eqref{EN:FUNCT}
provides a free boundary problem that is either one-phase (when~$\lambda_1=0$)
or two-phase (when~$\lambda_1\ne0$) and in which the interfacial
energy depends on the volume of the phases in a possibly nonlinear way,
which is described by the function~$\Phi$. The principal difference 
from the classical case is that the free boundary condition is determined 
by the weighted volumes of the phases and hence its Bernoulli constant is 
of global type
and varies from one minimizer to another.\medskip

More precisely, if the free boundary $\Gamma:=\fb u$ is a smooth hypersurface then 
\begin{equation}\label{eq-big-lamb:0}
|\na u^+(p)|^2-|\na u^-(p)|^2=\Lambda(p), \quad p\in \Gamma,\end{equation}
where 
\begin{equation}\label{eq-big-lamb}
\Lambda(p):=
\left[\lambda_2\partial_{r_2}\Phi
\left(\lambda_1
\int_\Om Q
\,\chi_{ \{ u<0\} },\;
\lambda_2
\int_\Om Q
\,\chi_{ \{ u>0\} }\right)\,
-\,
\lambda_1\partial_{r_1}\Phi
\left(\lambda_1
\int_\Om Q
\,\chi_{ \{ u<0\} },\;
\lambda_2
\int_\Om Q
\,\chi_{ \{ u>0\} }\right)\,
\right]\,Q(p).
\end{equation}
In the classical case  $\Lambda$ is a prescribed function
and only depends on the ambient space at a given point. 
Conversely, since in our setting $\Lambda$ depends in a
nonlinear fashion on global quantities
such as~$\mathcal M_1$ and~$\mathcal M_2$,  which in turn
depend on the solution,
it is natural to expect that the problem is going to be highly unstable
(other global free boundary conditions
arise in unstable free boundary problems as the one dealt with
in formula~(1.12)
in~\cite{DKV}). In particular, comparing~\eqref{BERc} with~\eqref{eq-big-lamb:0}
and~\eqref{eq-big-lamb}, we may consider the free boundary condition
of our problem as a nonlinear prescription of the pressure in terms
of the volume of the two phases of
the minimizer.\medskip

In our framework, the instability produced by the nonlinear superposition
of energies may be, in general, quite severe, and, in fact,
{\em the minimizers do not always exist}, as we will see in Section~\ref{sec:non}. 
Thus some structural assumptions are needed
in order to develop a meaningful theory.

\bigskip

Interesting examples of nonlinearities that we can take into account are
given by~$\Phi(r_1,r_2):=r_2^{\frac{n-1}{n}}$
and~$\Phi(r_1,r_2):=(r_1+r_2)^{\frac{n-1}{n}}$. We notice that this type
of nonlinearities provides 
a scaling which is 
naturally induced by the isoperimetric inequality.
For instance, the minimizers of the Athanasopoulos-Caffarelli-Kenig-Salsa 
functional \cite{ACKS}
\[\int_\Omega|\na u|^2+{\rm Per}_\Om(\{ u>0\})\]
generate energy levels that are
above the ones of our functional. If the phases have 
isoperimetric property then this levels coincide. 
In this sense,
the main difference between the energy functionals in~\eqref{EN:FUNCT}
and~\eqref{EN:FUNCT:AC}--\eqref{EN:FUNCT:ACF} 
lies in the different scaling of the
volume term. This can be seen, as a paradigmatic example, 
by looking at the functional 
\[\int_\Omega|\na u|^2+\Phi_0
\left(\mathcal L^n(\{ u>0\}\cap \Om)\right)\]
with
$$ c r^{\frac{n-1}n}\le \Phi_0(r)\le C r^{\frac{n-1}n},$$
for some~$C\ge c>0$.
We remark that different scalings in perimeter/volume terms
may cause instability phenomena in the corresponding minimization
arguments, namely a minimizer in a given domain is not necessarily
a minimizer in a smaller domain, see~\cite{DKV}.

Other cases of interest for the nonlinearity are the following ones:
\begin{itemize}
\item 
$\Phi$ only depends on the sum of the masses of the two phases,
namely when~$\Phi(r_1,r_2)=\tilde\Phi(r_1+r_2)$
(notice that in this case, condition~\eqref{LAMBDA12} is implied
by the two conditions~$\lambda_2\geq\lambda_1$ and~$\tilde\Phi'(r)\ge0$
for any~$r>0$).

\item $\Phi$ only depends on the the sum of different powers
of the masses of the two phases, namely
$$\Phi(r_1,r_2):= 
\frac{r_1^{1+\alpha}}{1+\alpha}+\frac{r_2^{1-\beta}}{1-\beta},$$
with~$\alpha\ge0$ and~$ \beta>0$.
Notice that in this case condition~\eqref{LAMBDA12} is satisfied
if~$\lambda_1$ is sufficiently small (possibly in dependence of~$\lambda_2$
and~$\lambda_\Omega$).
\end{itemize}
For different type of free boundary problems with
a specific piecewise linear function of a volume term see~\cite{NESTOR}.

\subsection{Main results}

The main results of this paper deal with the regularity of the minimizers
and of their free boundary. We stress that, in general,
minimizers {\em may not exist} and, when minimizers exist, their 
free boundary {\em may be irregular}. We will present some explicit
examples of these pathologies in Sections~\ref{sec:non} and~\ref{sec:irregular}.

In spite of these examples, under suitable structural assumptions,
a good regularity theory holds true. 

For this, recalling the notation in~\eqref{omegapiu},
we will suppose that, for a given minimizer~$u$,
\begin{equation}\label{VARPI}
{\mathcal{L}}^n
\big(\Omega^+(u)\big)\ge\varpi>0.\end{equation}
We will see that this assumption is not restrictive and it is
satisfied in all nontrivial cases (a precise statement will be given in Lemma~\ref{VAVA}).
Then, our
first result deals with the regularity of the gradient of the
minimizers in BMO spaces. To state it, 
for any~$x_0\in\Omega$ and for any~$\rho>0$,
we use the notation
\begin{equation}\label{recn}
(\na u)_{x_0,\rho}:=\fint_{B_\rho(x_0)}\nabla u(x)\,dx. 
\end{equation}
Then, we have the following:

\begin{thm}\label{THM:BMO}
Let~$u$ be a minimizer in~$\Omega$ for the functional~$J$
in~\eqref{EN:FUNCT}. 
Assume that~\eqref{LAMBDA12} and~\eqref{VARPI} hold true.

Then~$\nabla u\in BMO_{\rm loc}(\Om,\R^n)$.
More precisely, for any~$D\Subset\Om$, there exists~$C>0$,
possibly depending on~$\varpi$, $Q$, $\Om$ and~$D$, such that
$$ \sup_{ B_r(x_0)\subseteq D}
\fint_{B_r(x_0)}|\na u(x)-(\na u)_{x_0,r}|\,dx \le C. $$
\end{thm}

As a consequence of Theorem~\ref{THM:BMO}, we also
obtain the following result:

\begin{cor}\label{COR:5}
Let~$u$ be a minimizer in~$\Omega$ for the functional~$J$
in~\eqref{EN:FUNCT}. 
Assume that~\eqref{LAMBDA12} and~\eqref{VARPI} hold true. Then:
\begin{itemize}
\item $u$ is locally
log-Lipschitz continuous, namely it is continuous, with modulus 
of continuity bounded by~$\sigma(t)=t|\log t|$.
\item $u$ is harmonic in the set~$\Om^+(u)$.
\item For any~$D\Subset\Om$, there exists~$C>0$,
possibly depending on~$\varpi$, $Q$, $\Om$ and~$D$, such that
\begin{equation}\label{KA78} 
\left|\frac1{r^{n-1}}\int_{\partial B_r(x_0)}u\right|\le C\,r ,\end{equation}
for any~$x_0\in\Gamma$, as long as~$B_r(x_0)\subseteq D$.
\end{itemize}
\end{cor}

The regularity of the minimizers can also be improved, to reach the optimal
Lipschitz regularity, as given by the following result:

\begin{thm}\label{COR:7BIS}
Let $u$ be a minimizer in~$\Omega$ for the functional~$J$
in~\eqref{EN:FUNCT}. Assume that~\eqref{LAMBDA12} and~\eqref{VARPI} hold true. 
Then $u\in C^{0,1}_{loc}(\Omega)$. 
\end{thm}

We also deal with the geometric properties of the minimizers, obtaining
optimal quantitative results. In particular, we prove nondegeneracy
of minimizers and linear growth from the free boundary, as stated in the next result:

\begin{thm}\label{prop-nondeg}
Let $u$ be a minimizer in~$\Omega$ of the energy functional~$J$ 
in~\eqref{EN:FUNCT}, 
$\Om^+_0$ a connected component of the positivity set~$\Om^+(u)$, 
and~$x_0\in\partial\Omega^+_0$.

Assume that~\eqref{LAMBDA12} and~\eqref{VARPI} hold true.
Suppose that~$r>0$ is small enough such that~$B_r(x_0)\Subset \Omega$ and
\begin{equation}\label{JOHN}
{\mathcal{L}}^n\big(\Omega^+(u)\setminus B_r(x_0)\big)\ge\frac{\varpi}{2}
.\end{equation}
Assume also that
\begin{equation}\label{IPTH} \Theta:=
\inf_{ [\lambda_2 Q_1\varpi/2, \;\lambda_2\lambda_\Omega]} \Phi_0'>0.
\end{equation}
Then, for any $\kappa\in(0,1)$ there exists a positive constant~$c$, depending 
on $\varpi$, $\kappa$, $\Theta$ and~$Q$,  
such that if
$$\fint_{B_r(x_0)\cap \Om^+_0}u^2< c r^2,$$ then $u^+=0$ in
$B_{\kappa r}(x_0)\cap\Om_0^+$.

In particular, for any domain $D\Subset 
\Omega$ there exists a positive constant~$c$, 
depending only on $\varpi$, $Q$, $\Omega$ and
$\dist(D, \partial \Omega)$, such that 
\begin{equation}\label{PNOAL9}
\fint_{B_r(x_0)\cap \Om^+_0}u^2\ge c r^2,
\end{equation}
for any $x_0\in \partial \Om^+_0\cap D$ and~$r>0$, such that~$B_r(x_0)\Subset D$.
\end{thm}

Interestingly, the result in Theorem~\ref{prop-nondeg} holds true in any
connected component of the positivity set of the minimizers.

In this paper, we also establish
density results for minimizers, that can be of independent interest as well,
and that can be used to establish the minimizing properties
of the blow-up limits of the minimizers,
which indeed turn out to be minimizers of more classical free boundary problems.
In this setting, the result that we obtain is the following:

\begin{thm}\label{lem-ACF-blowup}
Let~$u$ be a minimizer in~$\Omega$ for the functional in~\eqref{EN:FUNCT} and let~$
u_0$ be the blow-up limit\footnote{A standard,
explicit definition of the blow-up limit
and the existence of such limit
is given in Proposition~\ref{tech-2}.}. 
Assume that~$Q$ is continuous at~$0$ and that~\eqref{LAMBDA12} and~\eqref{VARPI} hold true.

Then, for any fixed~$r>0$, we have that~$u_0$
is a minimizer of the functional
$$ J_0[w]:=\int_{B_r(x_0)} |\nabla w|^2 + \lambda_0 \,
{\mathcal{L}}^n\big( B_r(x_0)\cap\{w>0\}\big),$$
where
\begin{equation}\label{di u} \lambda_0 := \lambda_2\,Q(0)\,
\Phi_0'\left( \lambda_2\int_{\Omega} 
Q(x)\,\chi_{\{u>0\}}(x)\,dx\right).\end{equation}
\end{thm}

We stress
that the
quantity~$\lambda_0$
in~\eqref{di u} depends on the minimizer~$u$.
Furthermore,
it is useful to remark that, when~$\Phi_0$ is concave,
minimizers of~$J$ are also minimizers of the functional
\begin{eqnarray*}
&& J_\star [w]:=\int_{\Omega} |\nabla w|^2 + \lambda_\star \,\lambda_2\,
\int_{\Omega} Q(x)\,\chi_{\{w>0\}}(x)\,dx=
\int_{\Omega} |\nabla w|^2 +\lambda_\star \,{\mathcal{M}}_2(w),
\\
{\mbox{with }}&& \lambda_\star:=
\Phi_0'\left( \lambda_2\int_{\Omega} 
Q(x)\,\chi_{\{u>0\}}(x)\,dx\right)=\Phi_0'\left( {\mathcal{M}}_2(u)\right),
\end{eqnarray*}
namely the nonlinear free boundary problem in~\eqref{EN:FUNCT}
reduces to the classical ones in~\cite{AC, ACF} (with a coefficient
depending on the minimizer itself). Indeed, if~$u$ is a minimizer of~$J$
and~$v$ is a perturbation of~$u$, it follows that
\begin{eqnarray*}
0&\ge& J[u]-J[v]\\ &=&J_\star[u]-J_\star[v]
+\Phi_0\big({\mathcal{M}}_2(u)\big)
-\Phi_0\big({\mathcal{M}}_2(v)\big) +\lambda_\star
\left[ {\mathcal{M}}_2(v)-{\mathcal{M}}_2(u)\right].
\end{eqnarray*}
Since the concavity of~$\Phi_0$ implies that
$$ \Phi_0\big({\mathcal{M}}_2(v)\big)
-\Phi_0\big({\mathcal{M}}_2(u)\big)
\le \Phi_0'\big({\mathcal{M}}_2(u)\big)
\,\left[ {\mathcal{M}}_2(v)-{\mathcal{M}}_2(u)\right]=
\lambda_\star
\left[ {\mathcal{M}}_2(v)-{\mathcal{M}}_2(u)\right] ,$$
we thus obtain that~$0\ge J_\star[u]-J_\star[v]$, hence~$u$
is a minimizer for~$J_\star$.
\medskip

In our setting, we also obtain partial regularity results (valid in any
dimension) for free boundary points, as stated in the following result:

\begin{thm}\label{GMT-00}
Let~$u$ be a minimizer in~$\Omega$ for the functional~$J$
in~\eqref{EN:FUNCT}. 
Assume that~\eqref{LAMBDA12} and~\eqref{VARPI} are satisfied.

Then, the following statements hold true:
\begin{itemize}
\item[{(i)}] $\Delta u^+$ is a Radon
measure and, for any~$x\in\Gamma$ and any~$r>0$ such that~$B_{2r}(x)\subset\Omega$
and~${\mathcal{L}}^n
\big(\Omega^+(u)\setminus B_r(x)\big)\ge\varpi/2$, we have that
$$ \int_{B_r(x)}\Delta u^+\le\frac1{r}\int_{B_{2r}(x)}|\nabla u^+|.$$
\item[{(ii)}] For any subdomain~$D\Subset\Omega$ there exists~$r_0>0$
such that
$$ \int_{B_r(x)}\Delta u^+ \ge c r^{n-1},$$
for any~$ r \in(0, r_0)$, $x\in\Gamma$ and such that~$B_r(x)\subseteq D$,
for a suitable $c>0$.
\item[{(iii)}] If~$B_r\subseteq\Omega$, then
\begin{align}\label{90AKJA:A}
& \H^{n-1}\big(B_r\cap \partial\{u>0\}\big)<+\infty\\  \label{90AKJA:B}
{\mbox{and }}\quad&\H^{n-1}\big((B_r\cap\fb u)\setminus\fbr u\big)=0.
\end{align}
\end{itemize}
\end{thm}

In dimension~$2$, we also obtain a complete regularity theory for
the minimizers. This result goes as follows:

\begin{thm}\label{thm:reg}
Let~$n=2$.
Let~$\Phi_0$ be such that~$\Phi_0'>0$.
Assume also that~\eqref{VARPI} is satisfied. Then each free boundary point is critically flat and hence 
$\fb u$ is continuously differentiable.
\end{thm}

We stress that, in this paper,
the techniques that we develop
are strong enough to allow
a unified
treatment of the one and two phase cases simultaneously.

Furthermore, all the results of this paper
are valid in any dimension (with the only exception
of Lemma~\ref{lem-nax-verjin}
and Theorem~\ref{thm:reg}).

It is also worth to remark that many of the results
presented in this paper
are rather subtle to obtain, since they strongly rely
on some specific behavior of the nonlinearity,
and fail once these requirements are not met:
for instance,
the regularity result in Theorem~\ref{thm:reg}
fails to be true in the case of nonlinear functions~$\Phi_0$
which are constant on some interval (see Theorem~1.1
in~\cite{DKV}, where an example
of free boundary given by a singular cone in the plane
is constructed). In this sense, even the results
whose statements
``resemble'' cases already known
in the literature for the linear case
require a careful analysis
of the different setting and a precise determination
of appropriate structural assumptions.

\subsection{Organization of the paper}
The paper is organized as follows.
In Section~\ref{sec-exist} we prove the existence of 
minimizers. The proof is standard and is based on a 
semicontinuity argument and on a refinement 
of Egoroff's theorem for Sobolev functions.

In Section~\ref{sec:non}, an explicit example of volumetric 
function~$\Phi_0$ for 
which no solution exists is constructed. 
While~$\Phi_0(r)$ suffers a jump at~$r=1$, the non-existence 
is still surprising because it shows that, for such~$\Phi_0$, 
the set of admissible functions is not empty.  
In Section~\ref{sec:irregular} we construct another explicit example 
of~$\Phi_0$ that does not satisfy the structural assumption in~\eqref{LAMBDA12}:
in this case, minimizers do exist, 
but their free boundary is irregular.

That done, we begin to establish the basic properties of the 
minimizers in Section \ref{sec:BMO}. First a  lower bound for 
the positivity set~$\Omega^+(u)$ is proved 
for a suitable boundary condition. Then we show that 
the gradients of mimimizers are locally BMO functions,
that is we prove Theorem~\ref{THM:BMO}. This, 
in turn, implies that~$u$ is locally log-Lipschitz continuous, 
as given by Corollary~\ref{COR:5}. 
For the one-phase problem this immediately implies the
linear growth of $u$ away from free boundary. 
One of the by-products of the local BMO estimate is the coherent growth 
estimate~\eqref{KA78}. Using this and the Alt-Caffarelli-Friedman monotonicity 
theorem we prove the local Lipschitz regularity for 
the minimizers of the two-phase problem, as stated in Theorem~\ref{COR:7BIS}.

Then, we use a domain variation argument, to derive the 
nonlocal Bernoulli condition in Section~\ref{sec:FB-cond}. 

The non-degeneracy of minimizers given by 
Theorem~\ref{prop-nondeg}
is proved in Section~\ref{sec:nondeg}. 

In Section~\ref{sec:density} we show that for every ball $B$ 
centered at the free boundary there exists a 
smaller ball $B'\subset \Omega^+(u)\cap B$ 
such that $\mathcal L^n(B')\ge c\mathcal L^n(B)$,
for some universal constant $c>0$.  

Section \ref{sec:blowup} is devoted to the study of the properties of the blow-up limits and to the proof of Theorem~\ref{lem-ACF-blowup}.
In particular, we show that the blow-up limits
of the minimizers become global solutions for the 
Alt-Caffarelli-Friedman functional provided that $Q$ is continuous. 

In Section \ref{sec:partialreg}
we prove the partial regularity of the free boundary,
as given by Theorem~\ref{GMT-00}, namely, that the 
$\fb u$ is of locally finite perimeter and the reduced boundary 
has full $\mathcal H^{n-1}$ measure in $\fb u$. In particular, we show that 
the measure theoretic normal exists at $\mathcal H^{n-1}$ a.e. point of $\fb u$.
 
In Section \ref{sec:FBregularity2D} we prove that at the flat free boundary points the 
free boundary is regular and establish the full regularity of the free boundary in two dimensions, as stated in Theorem~\ref{thm:reg}.
 
\section*{Notation}
Let us fix some notation. 
\begin{itemize}
\item $\mathcal L^n$ is the $n$ dimensional Lebesgue measure. 
\item ${\mathcal{H}}^{n-1}$ is the $n-1$ dimensional Hausdorff measure.
\item $u^+(x):=\max\{u(x), 0\}$ and $u^-(x):=-\min\{u(x), 0\}$ are 
the positive and the negative parts of~$u$, respectively, 
so that $u=u^+-u^-$.
\item $\lambda_1\ge 0$ and $\lambda_2>0$
are given constants.
\item $\Omega^+(u):=\{x\in \Om: u(x)>0\}$ and 
$\Om^-(u):=\{x\in \Om: u(x)<0\}$ are the positivity and the negativity sets
of~$u$, respectively,
\item $\varpi>0$ is the constant providing a lower bound for the 
weighted volume of~$\Omega^+(u)$.
\item $\Gamma=\fb u$ is the free boundary.
\item The open balls are denoted by $B_r(x_0):=\{x\in \R^n {\mbox{ s.t. }} 
|x-x_0|<r\}$ and $B_r:=B_r(0)$.
\item $C^{0,1}_{loc}(\Omega)$ is the 
class of locally Lipschitz continuous functions in $\Om$.
\item The mean value integral is $\fint_E f=\frac1{\mathcal L^n(E)}\int_E f$.
\item Various universal constants are often denoted by $C$, for simplicity. 
\end{itemize}

\section{Existence and basic properties of minimizers}\label{sec-exist}

In this section we prove that there exists~$u\in \mathcal A$ 
minimizing~\eqref{EN:FUNCT}, 
where $\mathcal A$ is defined in \eqref{class}. 
We also 
show that under condition \eqref{LAMBDA12}, 
$u$ is globally subharmonic in $\Omega$.

\subsection{Existence of minimizers}

\begin{lem}\label{LE:EXI}
Fix~$
\bar u\in W^{1, 2}(\Om)$.
Then, there exists~$u\in\A$ such that
$$ J[u]\le J[v]$$
for any~$v\in\A$.
\end{lem}

\begin{proof} The proof
is a standard lower semicontinuity argument
(we give the details for the facility of the reader). We notice that
$$ {\mathcal{M}}(u_0)\le(\lambda_1+\lambda_2)\,Q_2\,{\mathcal{L}}^n(\Omega)<+\infty$$
and so~$J[\bar u]
<+\infty$. Now,
let~$u_k\in\A$ be a minimizing sequence.
We observe that~$\bar u\in \A$, hence for sufficiently large $k$ we may suppose that 
\begin{equation}\label{JU}
J[u_k]\le J[\bar u]<+\infty.\end{equation}
Set~$v_k:=u_k-\bar u\in W^{1,2}_0(\Om)$.
As a consequence of~\eqref{JU}, we have that
$$ \sup_{k\in\N} \| \nabla v_k\|_{L^2(\Omega,\R^n)} < +\infty,$$
and so, by Poincar\'e inequality, also
$$ \sup_{k\in\N} \| v_k\|_{L^2(\Omega)}<+\infty.$$
Therefore, up to a subsequence,~$v_k$
converges weakly to some~$v\in W^{1,2}_0(\Omega)$,
strongly in~$L^2(\Omega)$ and a.e. in~$\Om$.
Then, $\nabla v_k$ converges weakly to~$\nabla v$
in~$L^2(\Omega,\R^n)$. So, if we set~$u:=v+\bar u$,
we have that~$u\in\A$, $u_k\to u$ in~$L^2(\Omega)$ and a.e. in~$\Om$,
and~$\nabla u_k\to\nabla u$ weakly in~$L^2(\Omega,\R^n)$.
In particular,
\begin{equation}\label{KAH}
\liminf_{k\to+\infty} \int_\Om |\nabla u_k|^2 \ge \int_\Om |\nabla u|^2 .\end{equation}
Now we observe that
\begin{equation}\label{LAKTGAH}
\liminf_{k\to+\infty}\int_\Omega \lambda_2 Q(x)\,\chi_{\{ u_k>0 \}}(x)\,dx
\ge\int_{\Omega\setminus E_\e}
\lambda_2 Q(x)\,\chi_{\{ u>0 \}}(x)\,dx.
\end{equation}
For this, let $\e>0$ be fixed. Using the refinement of Egoroff's theorem for 
$W^{1,2}$ functions, it follows that there exists a subset $E_\e\subset \Omega$
such that $u_k\to u$ uniformly in $\Omega\setminus E_\e$ and ${\rm{cap}}_2(E_\e)<\e$
where $\rm{cap}_2$ is the 2-capacity. Thus
\[\liminf_{k\to+\infty}\int_\Omega \lambda_2 Q(x)\,\chi_{\{ u_k>0 \}}(x)\,dx
\ge\int_{\Omega\setminus E_\e}
\lambda_2 Q(x)\,\chi_{\{ u>0 \}}(x)\,dx-
\lambda_2Q_2|E_\e|.\]
Sending $\e\to 0$ the result in~\eqref{LAKTGAH} follows.

{F}rom~\eqref{LAKTGAH}, we deduce that
\begin{equation}\label{aUJAKA1qw1}\liminf_{k\to+\infty}
{\mathcal{M}}_2(u_k)=
\liminf_{k\to+\infty}\int_\Omega
\lambda_2 Q(x)\,\chi_{\{ u_k>0 \}}(x)\,dx
\ge\int_\Omega
\lambda_2 Q(x)\,\chi_{\{ u>0 \}}(x)\,dx = {\mathcal{M}}_2(u).\end{equation}
Furthermore, by~\eqref{UNASO},
$$ 0\le{\mathcal{M}}_1(u)
=\lambda_1\,\big(\lambda_\Omega-\lambda_{2}^{-1}{\mathcal{M}}_2(u)\big)$$
and therefore~${\mathcal{M}}_2(u) \in [0,\,\lambda_2\lambda_\Omega]$
and, similarly,~${\mathcal{M}}_2(u_k) \in [0,\,\lambda_2\lambda_\Omega]$.

Therefore, since, by~\eqref{LAMBDA12}, the function~$\Phi_0$
is nondecreasing in~$[0,\,\lambda_2\lambda_\Omega]$, we deduce from~\eqref{aUJAKA1qw1}
that
$$ \liminf_{k\to+\infty}\Phi_0\big( 
{\mathcal{M}}_2(u_k)\big) =
\Phi_0\left( \liminf_{k\to+\infty}
{\mathcal{M}}_2(u_k)\right) \ge \Phi_0\big({\mathcal{M}}_2(u) \big).$$
This and~\eqref{KAH} give that
$$ \liminf_{k\to+\infty} J[u_k]\ge J[u],$$
hence~$u$ is the desired minimizer.
\end{proof}

\subsection{Euler-Lagrange equations} 
We now state the basic properties of the minimizers.  The starting point is 
to derive the differential inequalities that the minimizers satisfy in $\Omega$.

\begin{lem}\label{LEMMA2}
Let~$u$ be a minimizer in~$\Omega$ for the functional~$J$
in~\eqref{EN:FUNCT} and suppose that~\eqref{LAMBDA12} holds true. 
Then~$u$ is subharmonic. 
\end{lem}

\begin{proof} We use some classical ideas in
Lemma~2.2 of~\cite{AC} and Theorem~2.3 in~\cite{ACF}, combining them here
with condition~\eqref{LAMBDA12}. For this, we consider a ball~$B\Subset\Omega$
and the function~$v$ which is harmonic in~$B$ and coincides with~$u$
in~$\Omega\setminus B$. We also take~$w:=\min\{u,v\}$.
Then, $w$ is an admissible competitor for~$u$ and therefore~$J[u]\le J[w]$, that is
\begin{equation} \label{PP1al}
I:=\int_{B}|\na u(x)|^2\,dx-
\int_{B}|\na w(x)|^2\,dx\le \Phi_0\big({\mathcal{M}}_2(w)\big)
-\Phi_0\big({\mathcal{M}}_2(u)\big).\end{equation}
On the other hand, if we set~$z:=\max\{ u-v,0\}$, we have that
\begin{eqnarray*}
&& I = \int_{B} \big(\na (u-w)(x)\big)\cdot
\big(\na (u+w)(x)\big)\,dx=\int_{B\cap\{ u>v\}} 
\big(\na (u-v)(x)\big)\cdot\big(\na (u+v)(x)\big)\,dx\\
&&\qquad =
\int_{B\cap\{ u>v\}} \big|\na (u-v)(x)\big|^2\,dx
+2\int_{B\cap\{ u>v\}} \big(\na (u-v)(x)\big)\cdot\nabla v(x)\,dx\\ &&\qquad= 
\int_{B\cap\{ u>v\}} \big|\na (u-v)(x)\big|^2\,dx=\int_B |\nabla z(x)|^2\,dx.
\end{eqnarray*}
Inserting this into~\eqref{PP1al}, we obtain that
\begin{equation} \label{PP1al:2}
\int_B |\nabla z(x)|^2\,dx
\le \Phi_0\big({\mathcal{M}}_2(w)\big)
-\Phi_0\big({\mathcal{M}}_2(u)\big).\end{equation}
Moreover, we have that~$w\le u$ and therefore~$\chi_{\{w>0\}}\le
\chi_{\{u>0\}}$. Accordingly, ${\mathcal{M}}_2(w)\le{\mathcal{M}}_2(u)$
and then, in light of~\eqref{LAMBDA12}, we obtain that~$\Phi_0\big({\mathcal{M}}_2(w)\big)
\le\Phi_0\big({\mathcal{M}}_2(u)\big)$. {F}rom this and~\eqref{PP1al:2} we deduce that~$z$
is constant in~$B$. Since~$z$ vanishes in~$\Omega\setminus B$, we
conclude that~$z$ vanishes in~$\Omega$ and therefore that~$u\le v$, which
establishes the desired result.
\end{proof}

\section{Non existence of minimizers for irregular nonlinearities}\label{sec:non}

In this section,
we observe that when the regularity and the structural assumptions
on~$\Phi_0$ are violated, minimizers may not exist.
To exhibit this phenomenon in an explicit and concrete
example, we consider the case in which~$\Omega:=(0,1)\subset\R$,
$\lambda_2:=Q:=1$,
$\bar u(x):=x$ for any~$x\in[0,1]$,
and
$$ \Phi(r_1,r_2) := \left\{
\begin{matrix}
r_2 & {\mbox{ if $r_2\in\left[0,\,\frac12\right]$, }}\\
&\\
\frac{5-2r_2}{8} & {\mbox{ if $r_2\in\left(\frac12,\,1\right)$, }}\\ & \\
1 & {\mbox{ if $r_2\in[1,+\infty)$. }}
\end{matrix}
\right. $$
Notice that, with this setting,
\begin{equation}\label{0sd8}
\Phi_0(r)= \left\{
\begin{matrix}
r & {\mbox{ if $ r\in\left[0,\,\frac12\right]$, }}\\ & \\
\frac{5-2r}{8} & {\mbox{ if $r\in\left(\frac12,\,1\right)$, }}\\ &\\
1 & {\mbox{ if $r\in[1,+\infty)$. }}
\end{matrix}
\right. \end{equation}
For this choice of $\Phi_0$, \emph{there exists no minimizer}~$u^*$ for the energy
functional in~\eqref{EN:FUNCT} with the condition 
that~$u^*-\bar u\in W^{1,2}_0(\Omega)$.

To see this, let us suppose, by contradiction, that such minimizer exists.
Then,
\begin{equation}\label{909051}
\int_\Omega |\dot u^*|^2 \le J[u^*]\le J[\bar u] = 1+ \Phi_0(1)=2.
\end{equation}
As a consequence,
\begin{equation*}
1-0=u^*(1)-u^*(0)=\int_{ (0,1)\cap\{u^*>0\}} \dot u^*\le
\sqrt{ \int_\Omega |\dot u^*|^2 } \,\sqrt{ {\mathcal{L}}^1
\big((0,1)\cap\{u^*>0\}\big) }\le \sqrt{ 2} \, \sqrt{ {\mathcal{M}}_2(u^*) }
\end{equation*}
and so
\begin{equation}\label{FR12b}
{\mathcal{M}}_2(u^*)\ge\frac12.
\end{equation}
We claim that
\begin{equation} \label{M091}
{\mbox{$\{u^*=0\}$ has positive measure.}}\end{equation}
To check this, we argue by contradiction and assume
that~${\mathcal{L}}^1( (0,1)\cap\{u^*=0\} )=0$, hence~$
{\mathcal{M}}_2(u^*)=1$. Consequently, since~$\bar u$ is
a minimizer for the Dirichlet energy in~$(0,1)$, we find that
\begin{equation}\label{MC0}
J[u^*] = \int_0^1 |\dot u^*|^2 + 1\ge
\int_0^1 |\dot{ \bar u}|^2 + 1 = 2.
\end{equation}
Now we define, for any~$\delta\in\left(0,\frac12\right)$,
$$ u_\delta(x):=
\left\{
\begin{matrix}
0 & {\mbox{ if $x\in[0,\,\delta]$, }}\\ &\\
\frac{x-\delta}{1-\delta}
& {\mbox{ if $x\in(\delta,\,1]$.}}
\end{matrix}
\right. $$
Then, there holds 
$$ J[u^*]\le J[u_\delta]
=\int_\delta^1 \left|\frac{1}{1-\delta}\right|^2 +\Phi_0 (1-\delta)
= \frac{1}{1-\delta}+ \frac{5-2(1-\delta)}{8}.$$
Accordingly, by taking~$\delta$ as small as we wish, we obtain
$$ J[u^*]\le 1 +\frac38.$$
This inequality is in contradiction with~\eqref{MC0}
and so it proves~\eqref{M091}.

In particular, from~\eqref{M091},
we can take a Lebesgue point~$p\in(0,1)$ for~$\{u^*=0\}$.
Thus, if~$\e>0$ is sufficiently small, we have that
\begin{equation}\label{MI01}
{\mathcal{L}}^1 \Big(
(p-\e,\,p+\e) \cap \{u^*=0\}\Big) \ge \e.\end{equation}
For small~$\e>0$, we can also suppose that~$(p-\e,\,p+\e)\subset(0,1)$.

Now we take~$\varphi\in C^\infty_0 ([-1,1])$
with~$\varphi>0$ in~$(-1,1)$ and~$|\dot\varphi|\le1$.
For any~$\e>0$, we define~$\varphi_\e(x):= \varphi\left( \frac{x-p}\e \right)$
and we remark that
\begin{equation}\label{01902}
\int_{p-\e}^{p+\e} |\dot \varphi_\e|^2 \le \frac{2}{\e}
.\end{equation}
Let also
$$ u_\e(x):= u^*(x)+\e^4 \varphi_\e(x).$$
Notice that~$u_\e\ge u^*$, and
\begin{equation}\label{dwuo22d}
\begin{split}
& {\mathcal{M}}_2(u_\e)-{\mathcal{M}}_2(u^*)=
{\mathcal{L}}^1 \big((0,1)\cap\{u_\e>0\}\big)- 
{\mathcal{L}}^1 \big((0,1)\cap\{u^*>0\}\big) \\
&\qquad\qquad= {\mathcal{L}}^1
\big(\{x\in (p-\e,\,p+\e) {\mbox{ s.t. }} u^*(x)=0\}\big)
\in[\e,\,2\e],
\end{split}\end{equation}
thanks to~\eqref{MI01}.
Notice also that, in view of~\eqref{FR12b} and~\eqref{M091},
$$ {\mathcal{M}}_2(u^*)\in \left[\frac12, \,1\right)$$
and so, if~$\e>0$ is sufficiently small, we deduce from~\eqref{dwuo22d} that also
$$ {\mathcal{M}}_2(u_\e)\in \left[\frac12, \,1\right).$$
Therefore, by~\eqref{0sd8} and~\eqref{dwuo22d},
\begin{equation}\label{0d61aqa1}
\Phi_0 \big({\mathcal{M}}_2(u_\e)\big)-\Phi_0\big({\mathcal{M}}_2(u^*)\big)=
-\frac{{\mathcal{M}}_2(u_\e)-{\mathcal{M}}_2(u^*)}{4}
\le -\frac{\e}4.
\end{equation}
On the other hand, recalling~\eqref{909051} and~\eqref{01902},
\begin{eqnarray*}
&& \int_\Omega |\dot u_\e|^2-
\int_\Omega |\dot u^*|^2 =
\int_{p-\e}^{p+\e} \big( 2\e^4\dot u^*\, \dot\varphi_\e+\e^8
|\dot\varphi_\e|^2\big) \\
&&\qquad\le 2\e^4 \,\sqrt{ \int_\Omega |\dot u^*|^2 }
\,\sqrt{ \int_{p-\e}^{p+\e} |\dot \varphi_\e|^2 }
+\e^8\int_{p-\e}^{p+\e}|\dot\varphi_\e|^2
\le 2\e^4 \,\sqrt{ 2}
\, \sqrt{ \frac{2}{\e} }
+2\e^7\le \e^3,
\end{eqnarray*}
as soon as~$\e>0$ is small enough. Using this and~\eqref{0d61aqa1},
we obtain that
$$ J[u_\e]-J[u^*] \le -\frac{\e}4 + \e^3 <0$$
if~$\e>0$ is small enough, which is a contradiction with the minimality
of~$u^*$. This shows that no minimizer exists in this case.

\section{Irregular free boundaries}\label{sec:irregular}

In this section, we would like to remark that if~$\Phi_0$
is not monotone, then there may exist minimizers 
whose free boundary is not regular, even in dimension~$2$
(therefore, the result in Theorem~\ref{thm:reg}
cannot be generalized to nonlinear problems
for which~$\Phi_0$ is not monotone).

To make an explicit example, we consider the case in which~$n=2$, $\Om:=B_1\subset\R^2$,
$\lambda_1:=\lambda_2:=Q:=1$ and
$\bar u(x):=x_1 x_2$. We also
define
\begin{eqnarray*}
&& c_1:= \int_{\partial B_1} \bar u^+,\\
&& c_2:=\frac{c_1}{2} \left[ \int_{B_1}|\nabla \bar u|^2+1\right]^{-1},\\
&& c_3:= 2+\frac{1}{4c_2}\\
{\mbox{and }} && c_\star := \min\left\{ \frac\pi4, \,\frac{c_1}{2c_3}\right\}.\end{eqnarray*}
We remark that~$c_\star<\pi/2$.
We consider a smooth function~$\phi_\star:[0,+\infty)\to[0,+\infty)$
with~$\phi_\star(0)=0$, $\phi_\star(r)>0$ for any~$r>0$, $\phi_\star(\pi/4)=2$ and
\begin{equation}\label{90Ka34zxA}
1=\phi_\star(\pi/2)=\min_{[c_\star,\,+\infty)} \phi_\star.
\end{equation}
Let also~$\Phi(r_1,r_2):=\phi_\star(r_2)$.
In this way, we have that~$\Phi_0(r)=\phi_\star(r)$
and we observe that all our structural assumptions on~$\Phi_0$ are satisfied in this case,
except the monotonicity.

We will show that
\begin{equation}\label{GOAL:J}
J[\bar u]=\min_{ u-\bar u\in W^{1,2}_0(B_1) } J[u],
\end{equation}
hence~$\bar u$ is a minimizer for~$J$ in~$B_1$ with respect to its own boundary values.
Interestingly, the set~$\{\bar u>0\}$ is in this case a singular cone,
which shows that the monotonicity assumption on~$\Phi_0$ cannot be dropped if one
wishes to prove that the free boundary of minimizers in the plane is smooth.

To prove~\eqref{GOAL:J}, we
argue as follows. Let~$u$ be such that~$u-\bar u\in W^{1,2}_0(B_1)$ and set~$v:=\min\{u,1\}$.
Then, $v-\bar u\in W^{1,2}_0(B_1)$, therefore~$v=\bar u$ along~$\partial B_1$ and thus,
if $\nu$ is the exterior normal of~$B_1$,
\begin{equation}\label{p0q7er}
\begin{split}
&c_1 = \int_{\partial B_1} \bar u^+ = \int_{\partial B_1} v^+=
\int_{\partial B_1} v^+ x\cdot\nu=
\int_{B_1} {\rm div}\, (v^+ x)
= \int_{B_1} (2v^+ +\nabla v^+\cdot x)
\\ &\qquad\qquad\le 2{\mathcal{L}}^2 (B_1\cap \{v>0\})+
\int_{B_1} |\nabla v^+|.
\end{split}
\end{equation}
Now we observe that
\begin{eqnarray*}
&&\int_{B_1}|\nabla v^+| = \int_{B_1} 2\cdot\frac{1}{2\sqrt{c_2}}\cdot\sqrt{c_2}\,|\nabla v^+|
\le \int_{B_1\cap\{v>0\}} \left( \frac{1}{4c_2} + c_2|\nabla v|^2\right)
\\ &&\qquad\le
\frac{1}{4c_2}\,{\mathcal{L}}^2 (B_1\cap \{v>0\})+
c_2 \int_{B_1} |\nabla u|^2 \le (c_3-2)\,{\mathcal{L}}^2 (B_1\cap \{v>0\})+c_2\,J[u].
\end{eqnarray*}
By plugging this information into~\eqref{p0q7er} and recalling that~$\{v>0\}=\{u>0\}$,
we obtain
$$ c_1\le c_3\,{\mathcal{L}}^2 (B_1\cap \{u>0\})+c_2\,J[u].$$
This implies that
either
\begin{equation}\label{p0q7er:1}
c_3\,{\mathcal{L}}^2 (B_1\cap \{u>0\})\ge \frac{c_1}{2}
\end{equation}
or
\begin{equation}\label{p0q7er:2}
c_2\,J[u]\ge \frac{c_1}{2}.
\end{equation}
If~\eqref{p0q7er:1} is satisfied, then
$$ {\mathcal{M}}_2(u)={\mathcal{L}}^2 (B_1\cap \{u>0\})\ge \frac{c_1}{2c_3}\ge
c_\star.$$ Consequently, by~\eqref{90Ka34zxA} and using that~$\pi={\mathcal{L}}^2(B_1)$,
we have
$$ \Phi_0 ( {\mathcal{M}}_2(u) )\ge 1=
\Phi_0(\pi/2) =\Phi_0 ( {\mathcal{M}}_2(\bar u) ).$$
Therefore, since~$\bar u$ is harmonic, we conclude that~$ J[u]\ge J[\bar u]$.
This proves~\eqref{GOAL:J} in this case, and we now consider the case in which~\eqref{p0q7er:2} holds true.

In this setting, we have that
$$ J[u]\ge \frac{c_1}{2 c_2}=\int_{B_1}|\nabla \bar u|^2+1=
\int_{B_1}|\nabla \bar u|^2+\Phi_0(\pi/2)=
J[\bar u].$$
This proves~\eqref{GOAL:J}, as\footnote{Let us remark
that the counterexamples discussed here
are based on the intuition that when~$\Phi_0'$ vanishes
at a minimizer, then the problem is related to
that of harmonic functions, and so the regularity
of the free boundary may be violated by looking at
harmonic functions
with irregular level sets.
Notice also that the condition~$\Phi_0'<0$ reduces to~$\Phi_0'>0$
by replacing $u$ by~$-u$.
We think that it is a very interesting problem
to further investigate the cases in which~$\Phi_0$
is not differentiable, or not continuous.}
desired.

\section{BMO gradient estimates and Lipschitz continuity of the minimizers}\label{sec:BMO}

In this section, we will prove that minimizers have gradient which is
locally in BMO and, as a consequence, we obtain an estimate for the integral averages of 
the minimizers. This method is structurally quite different from
the classical techniques in~\cite{AC, ACF}, which obtain
Lipschitz estimates in the linear case without using BMO estimates
on the gradient of the solution.

{F}rom now on we assume that~\eqref{LAMBDA12} and~\eqref{VARPI} hold true.

We remark that
condition~\eqref{VARPI} is satisfied in all nontrivial cases,
and then it links~$\varpi$ to quantities only depending on~$\Om$ and~$u_0$.
More precisely, condition~\eqref{VARPI}
is satisfied provided that~$\bar u^+=u^+\big|_{\partial\Omega}$ has some positive mass along the boundary
(and when this does not happen, the positive phase of the minimizer is trivial).
Indeed, we have the following observation:

\begin{lem}\label{VAVA}
Let~$u$ be a minimizer in~$\Omega$ for the functional~$J$
in~\eqref{EN:FUNCT}. Assume that~$\Omega$ has Lipschitz boundary and
that~$\bar u^+$ has some positive mass along~$\partial\Omega$.
Then~\eqref{VARPI} is satisfied, with
$$ \varpi:= 
\frac{1}{\displaystyle
C_\Omega \left( C_\Omega\,J[\bar u]
+2\|\bar u^+\|_{L^\infty(\partial\Omega)} \int_{\partial\Omega} \bar u^+\right)
}\,\left(\int_{\partial\Omega} \bar u^+\right)^2,$$
where~$C_\Omega>0$ is the trace constant for the domain~$\Omega$
for the embedding of~$W^{1,1}(\Omega)$ into~$L^1(\partial\Omega)$.
\end{lem}

\begin{proof} First of all,
Lemma~\ref{LEMMA2} gives that~$u$ is subharmonic,
hence so is~$u^+$. Therefore
$$ \|u^+\|_{L^\infty(\Omega)}=\|u^+\|_{L^\infty(\partial\Omega)}=
\|\bar u^+\|_{L^\infty(\partial\Omega)}.$$
Moreover, by the minimality of~$u$,
$$ \int_{\Omega^+(u)} |\nabla u|^2\le
\int_{\Omega} |\nabla u|^2\le J[u]\le J[\bar u].$$
Now, let~$\eta>0$, to be chosen appropriately.
By the trace inequality (see e.g.
Theorem~1(ii) on page~258 of~\cite{Evans}) and the observations above, we have that
\begin{eqnarray*}
&&
\int_{\partial\Omega} \bar u^+=
\| u^+\|_{L^1(\partial\Omega)}\le C_\Omega\,\|u^+\|_{W^{1,1}(\Omega)}
\\ &&\qquad= C_\Omega\int_{\Omega^+(u)} \big(|\nabla u|+u\big)
\le C_\Omega \int_{\Omega^+(u)} \left(\eta |\nabla u|^2+
\frac{1}{4\eta}+\|u^+\|_{L^\infty(\Omega)}\right)
\\ &&\qquad\le C_\Omega \left[ \eta\,J[\bar u] + 
\left(\frac{1}{4\eta}+\|\bar u^+\|_{L^\infty(\partial\Omega)}\right)\,{\mathcal{L}}^n
\big(\Om^+(u)\big)\right].
\end{eqnarray*}
Hence, we choose
$$ \eta:=\frac{\displaystyle\int_{\partial\Omega} \bar u^+}{2C_\Omega\,J[\bar u]}$$
and we obtain that
$$ C_\Omega\, \left(
\frac{C_\Omega\,J[\bar u]}{\displaystyle 2\int_{\partial\Omega} \bar u^+}
+\|\bar u^+\|_{L^\infty(\partial\Omega)}\right)\,{\mathcal{L}}^n\big(\Om^+(u)\big)
\ge \frac12\,\int_{\partial\Omega} \bar u^+,$$
which gives the desired result.
\end{proof}

Now we establish the BMO estimate claimed 
in Theorem~\ref{THM:BMO}:

\begin{proof}[Proof of Theorem~\ref{THM:BMO}]
Let~$B_r(x_0)\subseteq
D\Subset\Omega$. Without loss of generality,
we assume that~$u$ does not vanish identically,
hence~$\L^n(\Omega^+(u))>0$. 
We consider here the function~$v\in W^{1,2}(B_r(x_0))$ which solves
\begin{equation}\label{eq v}
\left\{
\begin{array}{lll}
\Delta v=0 &\text{in} \ B_r(x_0),\\
v=u &\text{on}\ \p B_r(x_0).
\end{array}
\right.
\end{equation}
We also extend~$v$ to be equal to~$u$ in~$\Omega\setminus B_r(x_0)$.
Since~$v\in\A$, we have that~$J[u]\le J[v]$ and therefore
\begin{eqnarray}\label{oj:9}
\int_{B_r(x_0)}\big(|\na u|^2-|\na v|^2\big) 
&\le& 
\Phi_0\big( {\mathcal{M}}_2(v)\big)- \Phi_0\big( {\mathcal{M}}_2(u)\big)\\\nonumber 
&=&
\Phi_0 \left(\lambda_2\int_{\Om\setminus B_r(x_0)}Q\I u+\lambda_2
\int_{B_r(x_0)}Q\I v\right)\\\nonumber
&&-\Phi_0 \left(\lambda_2\int_{\Om\setminus B_r(x_0)}Q\I u+\lambda_2\int_{B_r(x_0)}Q\I u\right).
\end{eqnarray}
Now we claim that
\begin{equation}\label{porvv r}
\int_{B_r(x_0)}\big(|\na u|^2-|\na v|^2\big)\le C_\star \,r^n,
\end{equation}
for some~$C_\star>0$ (independent of~$r$). 
To prove this, we first assume that~$r>0$ is so small that
\begin{equation}\label{c star}
\int_{\Omega\setminus B_r(x_0)}\I u \ge \frac{ \L^n(\Omega^+(u)) }{ 2 }
=:c_\star.\end{equation}
We stress that~$c_\star>0$ depends on~$\varpi$ (but not on~$r$).
We also set
$$ C_0:= \sup_{\xi \in [\lambda_2 c_\star, \lambda_2\lambda_\Omega)} \Phi_0'(\xi).$$
Let us  fix~$a\ge \lambda_2 c_\star$ and~$b$, $c\in [0,+\infty)$, 
with~$b\ge c$ and~$a+b<\lambda_2\lambda_\Omega$, then observe that
$$ \Phi_0(a+b) -\Phi_0 (a+c)
=\int_c^b \Phi_0'(a+\tau)\,d\tau\le  C_0\,(b-c)\le C_0\,b.
$$
On the other hand,
if~$a\ge \lambda_2 c_\star$ and~$b$, $c\in [0,+\infty)$, 
with~$b\le c$ and~$a+c<\lambda_2\lambda_\Omega$, then we have that
$$ \Phi_0(a+b) -\Phi_0 (a+c)\le0,$$
due to~\eqref{LAMBDA12}.
Therefore, for any~$a\ge \lambda_2 c_\star$ and~$b$, $c\in [0,+\infty)$, 
with~$a+b$, $a+c<\lambda_2\lambda_\Omega$, we get
$$ \Phi_0(a+b) -\Phi_0 (a+c)\le C_0\,b.$$
Utilizing this inequality with
\begin{equation}\label{NOTE}
\begin{split}
& a := \lambda_2\,\int_{\Omega\setminus B_r(x_0)}Q\I u,\\
& b := \lambda_2\,\int_{B_r(x_0)}Q\I v\\
{\mbox{and }}\quad &
c := \lambda_2\,\int_{B_r(x_0)}Q\I u\end{split}
\end{equation}
yields
\begin{eqnarray*} &&\Phi_0
\left(\lambda_2\int_{\Om\setminus B_r(x_0)}Q\I u+
\lambda_2\int_{B_r(x_0)}Q\I v\right)-\Phi_0
\left(\lambda_2\int_{\Om\setminus B_r(x_0)}Q\I u+
\lambda_2\int_{B_r(x_0)}Q\I u\right)
\\ && \qquad\le C_0\,\lambda_2\,\int_{B_r(x_0)}Q \I v
\le C_\star \,r^n,\end{eqnarray*}
for some~$C_\star>0$ (possibly depending on~$\varpi$ in~\eqref{VARPI}).
Plugging this  into~\eqref{oj:9} we see  that~\eqref{porvv r}
is satisfied
if~$r$ is chosen so small to fulfill~\eqref{c star} (say~$r\in[0,r_0]$).

Now we complete the proof of~\eqref{porvv r} when~$r>r_0$.
In this case, we use the notation 
in~\eqref{NOTE} and we claim that
\begin{equation}\label{NMA}
\Phi_0(a+b)\le \hat C\, (b+1),
\end{equation}
for some~$\hat C>0$, possibly depending on~$r_0$, $Q$, $\lambda_2$ and~$\Om$.
To this goal, we define
$$ a_0:= \lambda_2\int_{\Om\setminus B_{r_0}(x_0)}Q \;\;{\mbox{ and }}\;\;
b_0:=a+b-a_0.$$
Notice that the condition~$r>r_0$ implies that~$a_0\ge a$ and so~$b_0\le b$.
We distinguish two cases, either~$b_0\le0$ or~$b_0>0$.
If~$b_0\le0$, we use~\eqref{LAMBDA12} and we obtain
$$ \Phi_0(a+b)=\Phi_0(a_0+b_0)\le \Phi_0(a_0).$$
This implies~\eqref{NMA} in this case.
If instead~$b_0>0$, we have
$$ \Phi_0(a+b)=\Phi_0(a_0+b_0)=
\int_0^{b_0}\Phi_0'(a_0+\tau)\,d\tau+\Phi_0(a_0)
\le \sup_{\xi\in [a_0,\lambda_2\lambda_\Omega)}\Phi'_0(\xi)\,b_0
+\Phi_0(a_0)\le
\sup_{\xi\in [a_0,\lambda_2\lambda_\Omega)}\Phi'_0(\xi)\,b
+\Phi_0(a_0),$$
which completes the proof of~\eqref{NMA}.

Now we observe that~$ b\le C'\,r^n$,
for some~$C'>0$. This fact, 
together with~\eqref{NMA} and the assumption that~$r>r_0$,
gives that~$\Phi_0(a+b)\le \tilde C\,r^n$.
Using this and
the estimate in~\eqref{oj:9}, we find that
$$\int_{B_r(x_0)}\big(|\na u|^2-|\na v|^2\big)
\le\Phi_0
\left(\lambda_2\int_{\Om\setminus B_r(x_0)}Q\I u+\lambda_2
\int_{B_r(x_0)}Q\I v\right)=\Phi_0(a+b)\le \tilde C\,r^n,
$$
which establishes~\eqref{porvv r} also when~$r>r_0$.

In addition, by~\eqref{eq v},
\begin{eqnarray*}
\int_{B_r(x_0)}\big(|\na u|^2-|\na v|^2\big)=
\int_{B_r(x_0)}\big(|\na u|^2-|\na v|^2 +2\nabla v\cdot\nabla(v-u)\big)
=
\int_{B_r(x_0)}|\na u-\na v|^2.
\end{eqnarray*}
This and~\eqref{porvv r} yield that
\begin{equation}\label{porvv r:2}
\int_{B_r(x_0)} |\na u-\na v|^2\le C_\star \,r^n.
\end{equation}
Now we use some techniques developed in~\cite{DK}.
We recall the notation in~\eqref{recn}
and we observe that, using H\"older inequality,
\begin{equation}\begin{split}\label{3.2bis}
& \int_{B_r(x_0)}|(\na v)_{x_0,r}-(\na u)_{x_0,r}|^2 =
\int_{B_r(x_0)}\left|\frac{1}{\L^n(B_r(x_0))}\left(\int_{B_r(x_0)}\na v-\na u\right)\right|^2 \\
&\quad\qquad\le \frac{1}{\L^n(B_r(x_0))}\left(\int_{B_r(x_0)}|\nabla v-\nabla u|\right)^2
\le \int_{B_r(x_0)}|\nabla v-\nabla u|^2.
\end{split}\end{equation}
Furthermore, we recall
the following Campanato growth  type estimate (see e.g.
Theorem 5.1 in \cite{DiB-M}), valid for any~$0<r<R$,
\begin{equation}\label{Manf-camp-0}
\int_{B_r(x_0)}|\na v-(\na v)_{x_0,r}|^2\le C\,
\left(\frac rR\right)^{n+\alpha}\int_{B_R(x_0)}|\na v-(\na v)_{x_0,R}|^2,
\end{equation}
for suitable~$\alpha\in(0,1)$ and~$C>1$.

Now, using \eqref{3.2bis}
and possibly allowing~$C$ to be a universal constant varying from line to line, we have
\begin{eqnarray*}
&& \int_{B_r(x_0)}|\na u-(\na u)_{x_0,r}|^2\\
&\le& 
C\,\left[ \int_{B_r(x_0)}|\na u-\na v|^2+
\int_{B_r(x_0)}|\na v-(\na v)_{x_0,r}|^2
+\int_{B_r(x_0)}|(\na v)_{x_0,r}-(\na u)_{x_0,r}|^2\right] \\
&\le& 
C\,\left[ \int_{B_r(x_0)}|\na u-\na v|^2+
\int_{B_r(x_0)}|\na v-(\na v)_{x_0,r}|^2\right].\end{eqnarray*}
So, using~\eqref{Manf-camp-0}, 
\begin{equation}\label{11AOJM}
\int_{B_r(x_0)}|\na u-(\na u)_{x_0,r}|^2
\le 
C\,\left[ \int_{B_r(x_0)}|\na u-\na v|^2+
\left(\frac rR\right)^{n+\alpha}\int_{B_R(x_0)}|\na v-(\na v)_{x_0,R}|^2
\right].\end{equation}
Now we remark that
\begin{eqnarray*}
&&
\int_{B_R(x_0)}|\na v-(\na v)_{x_0,R}|^2\\
&\le& C\,\left[ 
\int_{B_R(x_0)}|\na v-\na u|^2+
\int_{B_R(x_0)}|\na u-(\na u)_{x_0,R}|^2
+\int_{B_R(x_0)}|(\na u)_{x_0,R}-(\na v)_{x_0,R}|^2
\right] \\
&\le& C\,\left[ 
\int_{B_R(x_0)}|\na v-\na u|^2+
\int_{B_R(x_0)}|\na u-(\na u)_{x_0,R}|^2\right],
\end{eqnarray*}
where~\eqref{3.2bis} has been used once again.

Let us plug this into~\eqref{11AOJM},
and recall that~$r\le R$. We exploit~\eqref{porvv r:2}, and conclude that 
\begin{eqnarray*}
&&\int_{B_r(x_0)}|\na u-(\na u)_{x_0,r}|^2\\
&\le&
C\,\left[ \int_{B_r(x_0)}|\na u-\na v|^2+
\left(\frac rR\right)^{n+\alpha}
\int_{B_R(x_0)}|\na v-\na u|^2+\left(\frac rR\right)^{n+\alpha}
\int_{B_R(x_0)}|\na u-(\na u)_{x_0,R}|^2
\right]
\\ &\le& 
C\,\left[ \int_{B_R(x_0)}|\na u-\na v|^2+
+\left(\frac rR\right)^{n+\alpha}
\int_{B_R(x_0)}|\na u-(\na u)_{x_0,R}|^2
\right]\\
&\le& CR^n + C\,\left(\frac rR\right)^{n+\alpha}
\int_{B_R(x_0)}|\na u-(\na u)_{x_0,R}|^2
.\end{eqnarray*}
Therefore, defining
$$ \psi(r):=\sup\limits_{t\leq r} \int\limits_{B_t(x_0)}|\na u-(\na u)_{x_0,t}|^2,$$
we have that
$$\psi(r)\le CR^n+C\left(\frac rR\right)^{n+\alpha}\psi(R).$$
Thus, by Lemma~2.1 in Chapter~3  
of~\cite{Giaq}, 
we  conclude that there exist $c>0$ 
and $R_0>0$ such that
$$\psi(r)\leq Cr^n\left(\frac{\psi(R)}{R^n}+1\right)$$
for all $r\le R\le R_0$,
and hence
$$ \int_{B_r(x_0)}|\na u-(\na u)_{x_0,r}|^2 \le C r^n$$
for some tame constant $C>0$.

Therefore, by H\"older inequality,
$$ \fint_{ B_r(x_0)}|\na u-(\na u)_{x_0,r}|
\le \sqrt{ \fint_{ B_r(x_0)}|\na u-(\na u)_{x_0,r}|^2 }\le C,$$
up to renaming constants, as desired.
\end{proof}

Exploiting Theorem~\ref{THM:BMO}, 
we can now prove Corollary~\ref{COR:5}, by arguing as follows:

\begin{proof}[Proof of Corollary~\ref{COR:5}] By
Theorem~\ref{THM:BMO}, we have that~$\nabla u\in L^q_{\rm loc}(\Omega)$,
for any $1<q<+\infty$. 
Hence~$u$ is continuous.
The modulus of continuity $\sigma$ follows as in \cite{Morrey} and \cite{ACF}.
Therefore, $\Om^+(u)$ is open and thus, if~$x_0\in\Om^+(u)$,
there exists~$r>0$ such that~$\overline{B_r(x_0)}\subset \Om^+(u)$.
Consequently,
$$ m:=\min_{\overline{B_r(x_0)} } u >0.$$
Let now~$\phi\in C^\infty_0 (B_r(x_0))$, $\epsilon\in\R$
and~$u_\epsilon:=u+\epsilon\phi$. If~$|\epsilon|< \frac{m}{
1+\|\phi\|_{L^\infty(\R^n)}}$,
we have that~$u_\epsilon>0$ in~$B_r(x_0)$ and so
$$ \Om^+(u_\epsilon)=\Om^+(u).$$
This implies that, for small~$\epsilon$,
$$ 0\le J[u_\epsilon] - J[u] =\int_\Om \big(|\nabla u_\epsilon|^2 
-|\nabla u|^2\big)$$
and therefore
$$ \int_\Om \nabla u\cdot\nabla\phi=0,$$
which shows that~$u$ is harmonic in~$B_r(x_0)$, as desired.

Now we prove~\eqref{KA78}. For this we observe that, by the continuity of~$u$, it follows that
$$ \lim_{r\to 0}\fint_{\partial B_r(x_0)}u=0 \quad \mbox{ for any } 
x_0\in \Gamma.$$
Therefore
\begin{equation}\label{7bis}
\begin{split}
\frac{1}{r^{n-1}} \int_{\partial B_r(x_0)} u\;&=\int_0^r 
\frac{d}{dt} \left(
\frac{1}{t^{n-1}} \int_{\partial B_t(x_0)} u(x)\, 
d{\mathcal{H}}^{n-1}(x)\right) \,dt
\\&=
\int_0^r 
\frac{d}{dt} \left( \int_{\partial B_1} u(x_0+t\omega)\, 
d{\mathcal{H}}^{n-1}(\omega)\right) \,dt\\&= 
\int_0^r\left(
\int_{\partial B_1} \nabla u(x_0+t\omega)\cdot\omega\, d{\mathcal{H}}^{n-1} (\omega) \right)\,dt\\
&=\int_0^r\left( \frac{1}{t^{n}}
\int_{\partial B_t} \nabla u(x_0+\nu)\cdot\nu\,d{\mathcal{H}}^{n-1} (\nu) \right)\,dt
\\&=\int_0^r\left( \frac{1}{t^{n}} \frac{d}{dt}\left(
\int_{B_t} \nabla u(x_0+\nu)\cdot\nu\,d{\mathcal{H}}^{n-1} (\nu)\right) \right)\,dt
\\ &=
\int_0^r\left( \frac{1}{t^{n-1}} \frac{d}{dt}\left(
\int_{B_t(x_0)} \nabla u(x)\cdot\frac{x-x_0}{|x-x_0|}\,d{\mathcal{H}}^{n-1} (x)\right) \right)\,dt
.\end{split}
\end{equation}
Now, 
we notice that
\begin{equation}\label{odd}
\int_{B_\epsilon(x_0)} (\na u)_{x_0,\epsilon}
\cdot\frac{x-x_0}{|x-x_0|}\,dx=0\end{equation}
by odd symmetry, for any~$\epsilon>0$. In consequence of this, we have that
\begin{eqnarray*}\left|
\frac{1}{\epsilon^{n-1}}
\int_{B_\epsilon(x_0)}\nabla u(x)\cdot\frac{x-x_0}{|x-x_0|}\,dx \right|&=&
\left|\frac{1}{\epsilon^{n-1}}\int_{B_\epsilon(x_0)}\left(\nabla u(x)-
(\na u)_{x_0,\epsilon}
\right)\cdot\frac{x-x_0}{|x-x_0|}\,dx \right|\\
&\le &\epsilon\left(\frac{1}{\epsilon^{n}}\int_{B_\epsilon(x_0)}\left|
\nabla u(x)- (\na u)_{x_0,\epsilon}\right|\,dx\right)\longrightarrow 0,
\end{eqnarray*}
as~$\epsilon\to0$, thanks to the BMO estimate in Theorem~\ref{THM:BMO}.

Thus, an integration by parts gives that
\begin{eqnarray*}
&& \int_0^r\left( \frac{1}{t^{n-1}} \frac{d}{dt}\left(
\int_{B_t(x_0)} \nabla u(x)\cdot\frac{x-x_0}{|x-x_0|}\,
d{\mathcal{H}}^{n-1} (x)\right) \right)\,dt
\\ &=&
\frac{1}{r^{n-1}} 
\int_{B_r(x_0)} \nabla u(x)\cdot\frac{x-x_0}{|x-x_0|}\,
d{\mathcal{H}}^{n-1} (x)
+(n-1)
\int_0^r \frac{1}{t^{n}}
\int_{B_t(x_0)} \nabla u(x)\cdot\frac{x-x_0}{|x-x_0|}\,
d{\mathcal{H}}^{n-1} (x) \,dt.\end{eqnarray*}
So, recalling~\eqref{7bis}
and using again~\eqref{odd}, we obtain that
\begin{eqnarray*}
\frac{1}{r^{n-1}} \int_{\partial B_r(x_0)} u
&=&
\frac{1}{r^{n-1}}
\int_{B_r(x_0)} \big( \nabla u(x)-(\na u)_{x_0,r}\big)
\cdot\frac{x-x_0}{|x-x_0|}\,
d{\mathcal{H}}^{n-1} (x)
\\ &&\qquad+(n-1)
\int_0^r \frac{1}{t^{n}}
\int_{B_t(x_0)} \big(\nabla u(x)-(\na u)_{x_0,t}\big)
\cdot\frac{x-x_0}{|x-x_0|}\,
d{\mathcal{H}}^{n-1} (x) \,dt.\end{eqnarray*}
Therefore, the BMO estimate in Theorem~\ref{THM:BMO} yields
the desired result in~\eqref{KA78}.\end{proof}

As customary, one can deduce from the integral estimate
in~\eqref{KA78} a linear growth from the free boundary. We give
the details for convenience, starting from the one-phase case:

\begin{cor}\label{COR:6}
Let~$u\ge0$ be a minimizer in~$\Omega$ for the functional~$J$
in~\eqref{EN:FUNCT} and let~$D\Subset\Om$.
Let~$\varpi>0$ and assume that~${\mathcal{L}}^n
\big(\Omega^+(u)\big)\ge\varpi$.

Then, there exists~$C>0$,
possibly depending on~$\varpi$, $Q$, $\Om$ and~$D$, such that
\begin{equation*} 
u(x)\le C\,{\rm dist}(x,\Gamma) ,\end{equation*}
for any~$x\in D$ for which~$B_{ 2{\rm dist}(x,\Gamma) }(x)\subset D$.
\end{cor}

\begin{proof}
Let~$d$ be the distance of~$x$ to~$\Gamma$.
Let~$x_0\in \overline{ B_{d}(x)}\cap \Gamma$.
Then, we can use~\eqref{KA78} and obtain that, for any~$\rho\in (0,2d)$
$$ \int_{\partial B_\rho(x_0)} u\le C\rho^n,$$
for some~$C>0$. So, we integrate this inequality in~$\rho\in (0,2d)$
and we find that
\begin{equation}\label{FI:1} \int_{B_{2d}(x_0)} u\le Cd^{n+1},\end{equation}
up to renaming~$C>0$.

On the other hand, since $u$ is harmonic in~$B_d(x)$,
thanks to Corollary~\ref{COR:5}, we have that
\begin{equation}\label{FI:2} u(x)= \fint_{B_d(x)} u.\end{equation}
Notice now that~$B_d(x)\subseteq B_{2d}(x_0)$, hence we deduce from~\eqref{FI:1}
and~\eqref{FI:2} that
$u(x)\le Cd$, as desired.
\end{proof}

From Corollary~\ref{COR:6} one can deduce that  that~$u$
is Lipschitz continuous, as stated in the next result for completeness:

\begin{cor}\label{COR:7}
Let~$u\ge0$ be a minimizer in~$\Omega$ for the functional~$J$
in~\eqref{EN:FUNCT} and let~$D\Subset\Om$.
Let~$\varpi>0$ and assume that~${\mathcal{L}}^n
\big(\Omega^+(u)\big)\ge\varpi$.

Then, there exists~$C>0$,
possibly depending on~$\varpi$, $Q$, $\Om$ and~$D$, such that
\begin{equation*} 
\sup_{x\in D} |\nabla u(x)|\le C .\end{equation*}
\end{cor}

The proof of Corollary~\ref{COR:7} is by now standard (see e.g.
Theorem~5.3 in~\cite{ACF}).

The Lipschitz estimate in Corollary~\ref{COR:7} is optimal,
since the solutions have linear growth from the free boundary, as
stated in the following result:

\begin{lem}\label{BY BELOW}
Let~$u$ be a minimizer in~$\Omega$ for the functional~$J$
in~\eqref{EN:FUNCT} and let~$D\Subset\Om$.

Let~$d>0$ and suppose that~$B_d\subseteq\Om^+(u)$.

Let~$\bar\omega\ge\varpi>0$ and assume that
\begin{equation}\label{IOTA:0}
{\mathcal{L}}^n
\big(\Omega^+(u)\big)\in[\varpi,\bar\omega].\end{equation}
Assume also that
\begin{equation}\label{IOTA}
\iota:=\inf_{ r\in [
\lambda_2 Q_1
\varpi/2, \;
2\lambda_2 Q_2
\bar\omega]} \Phi_0'(r)>0.
\end{equation}
Then, there exist~$d_0$, $c_0>0$, possibly depending on~$\varpi$, 
$\bar\omega$, $Q$, $\lambda_1$, $\lambda_2$ and~$\Omega$, such that if~$d\in(0,d_0)$ we have that
$$u(0)\ge c_0\,\iota\, d.$$
\end{lem}

\begin{proof} We let
$$ C_0:=
\lambda_2 \int_{\Omega\setminus B_{d/2}} Q\chi_{\{u>0\}}.$$
We take~$d_0>0$ small enough such that
$$ C_0\ge 
\lambda_2 \,Q_1\,\frac{\varpi}2.$$
In addition, we have that
$$ C_0\le 
\lambda_2 \,Q_2\,\bar\omega.$$
Then, for any~$a$, $b\in [0, \,
\lambda_2 Q_2\bar\omega]$, with~$a\ge b$, we have that
\begin{equation}\label{IOTA:CON}
\Phi_0(C_0+a)-\Phi_0(C_0+b)
\ge \iota\,(a-b),
\end{equation}
thanks to~\eqref{IOTA}.

Also, from Corollary~\ref{COR:5}, we know that~$u$ is harmonic in $B_d$
and so, by Harnack inequality,
\begin{equation}\label{HAH}
\sup_{B_{d/2}} u\le \bar C \inf_{B_{d/2}} u\le \bar C u(0),\end{equation}
for some~$\bar C>0$.

We take~$\psi_0\in C^\infty(\R^n)$, such that~$\psi_0=0$ in~$B_{1/4}$,
$\psi_0=1$ on~$\partial B_{1/2}$ and~$|\nabla \psi_0|\le 10$.
We also set
$$ \psi(x):= 2\bar C \,u(0)\,\psi_0\left( \frac{x}{d}\right).$$
Notice that
\begin{equation}\label{GA28}
|\nabla \psi|\le \frac{20\bar C\,u(0)}{d}.\end{equation}
If~$x\in B_{d/2}$, we define
$$ v(x):=\min\{ u(x),\,\psi(x)\}.$$
Notice that if~$x\in\partial B_{d/2}$, then
$$\psi(x)=2\bar C \,u(0)\ge u(x),$$
thanks to~\eqref{HAH}, hence~$v=u$ on~$\partial B_{d/2}$.
Therefore, we extend~$v(x):=u(x)$ for any~$x$ outside~$B_{d/2}$, and we have that
\begin{equation}\label{GA29}
\begin{split}
&0\le J[v]-J[u]
= \int_\Omega |\nabla v|^2 - \int_\Omega |\nabla u|^2
+\Phi_0({\mathcal{M}}_2(v))-\Phi_0({\mathcal{M}}_2(u))\\
&\qquad\qquad=
\int_{B_{d/2}\cap\{ \psi>u\}} |\nabla \psi|^2 - \int_{B_{d/2}\cap\{\psi>u\}} |\nabla u|^2
+\Phi_0({\mathcal{M}}_2(v))-\Phi_0({\mathcal{M}}_2(u)).
\end{split}
\end{equation}
Now, from~\eqref{GA28}, we have that
\begin{equation}\label{GA30}
\int_{B_{d/2}\cap\{ \psi>u\}} |\nabla \psi|^2 - \int_{B_{d/2}\cap\{\psi>u\}} |\nabla u|^2
\le
\int_{B_{d/2}} |\nabla \psi|^2\le
\frac{400\bar C^2\,{\mathcal{L}}^n(B_{d/2})\,u^2(0)}{d^2}.
\end{equation}
On the other hand
\begin{eqnarray*}
\Phi_0({\mathcal{M}}_2(u))-\Phi_0({\mathcal{M}}_2(v))&=&
\Phi_0\left( \lambda_2 \int_\Omega Q\chi_{\{u>0\}}\right)-
\Phi_0\left( \lambda_2 \int_\Omega Q\chi_{\{v>0\}}\right) \\
&=&
\Phi_0\left( C_0+\lambda_2 \int_{B_{d/2}} Q\chi_{\{u>0\}}\right)-
\Phi_0\left( C_0+\lambda_2 \int_{B_{d/2}} Q\chi_{\{v>0\}}\right).
\end{eqnarray*}
Notice that the quantity 
$$ \lambda_2 \int_{B_{d/2}} Q\chi_{\{u>0\}}+
\lambda_2 \int_{B_{d/2}} Q\chi_{\{v>0\}}$$
is small if~$d$ is small enough, and so we can apply~\eqref{IOTA:CON}
with~$a:= C_0+\lambda_2 \int_{B_{d/2}} Q\chi_{\{u>0\}}$
and~$b:= C_0+\lambda_2 \int_{B_{d/2}} Q\chi_{\{v>0\}}$. In this way,
we find that
\begin{eqnarray*}
\Phi_0({\mathcal{M}}_2(u))-\Phi_0({\mathcal{M}}_2(v))&\ge&\iota
\lambda_2 \int_{B_{d/2}} Q \,\big(\chi_{\{u>0\}}-\chi_{\{v>0\}}\big)\\
&=&
\lambda_2 \int_{B_{d/2}\cap\{ u>\psi\}} Q \,\big(\chi_{\{u>0\}}-\chi_{\{\psi>0\}}\big).
\end{eqnarray*}
Notice also that in~$B_{d/4}$ we have that~$\psi=0<u$,
hence we conclude that
$$ \Phi_0({\mathcal{M}}_2(u))-\Phi_0({\mathcal{M}}_2(v))\ge
\lambda_2 \int_{B_{d/4}} Q \,\big(\chi_{\{u>0\}}-\chi_{\{\psi>0\}}\big)\ge
\lambda_2 \,Q_1\,{\mathcal{L}}^n (B_{d/4}).$$
Now, plugging this and~\eqref{GA30} into~\eqref{GA29} we infer 
$$ \frac{400\bar C^2\,{\mathcal{L}}^n(B_{d/2})\,u^2(0)}{d^2}\ge
\lambda_2 \,Q_1\,{\mathcal{L}}^n (B_{d/4}),$$
which implies the desired result.
\end{proof}

The Lipschitz regularity for the pure two-phase problem,
as stated in Theorem~\ref{COR:7BIS}, can be deduced from the 
BMO estimate, giving coherent growth for $u^+$ 
and $u^-$, and the classical Alt-Caffarelli-Friedman
monotonicity formula. The details go as follows:

\begin{proof}[Proof of Theorem~\ref{COR:7BIS}]
First we observe that, from \eqref{KA78}, it follows that 
\begin{equation}\label{wahlle}
\left|\frac{1}{r^{n-1}}\int_{\partial B_r(x)} u^+
-\frac{1}{r^{n-1}} \int_{\partial B_r(x)} u^-\right|\le C,
\end{equation}
for any~$x\in\Gamma$ and~$r>0$ such that~$B_r(x)\subseteq D\Subset\Omega$.

We recall now the Alt-Caffarelli-Friedman monotonicity formula \cite{ACF}:   
Let $w^+,w^-$ be two continuous, nonnegative subharmonic 
functions in~$B_1$, with~$w_1w_2=0$, $w_1(0)=w_2(0)=0$.
Then, for any~$x_0\in\Gamma$, 
\begin{equation*}
\Phi(r, w_1, w_2):=\frac{1}{r^4}\int_{B_r(x_0)}
\frac{|\nabla w_1(x)|^2}{|x-x_0|^{n-2}}\,dx
\int_{B_r(x_0)}\frac{|\nabla w_2(x)|^2}{|x-x_0|^{n-2}}\,dx
\end{equation*}
is a monotone increasing function of~$r\in(0,1)$, and 
\begin{equation}\label{5.27}
\Phi(1, w_1, w_2)\leq C\left(1+\int_{B_1} w_1^2+ \int_{B_1}w_2^2\right),
\end{equation}
with some universal constant $C>0$.

In what follows, we will apply this theorem with $w_1:=u^+, w_2:=u^-$.

To fix the ideas we assume that $B_1\subset D$. 
Moreover, we take~$x_0\in D$ such that $u(x_0)>0$ and let 
$x\in\Gamma=\p\{u>0\}$ be the closest point to $x_0$, that is 
$\dist(x_0, \Gamma)=|x_0-x|$.  

Setting~$\rho:=|x-x_0|$, 
we suppose that~$u(x_0)\geq M\rho>0$, for some large $M>0$.
Hence, applying the Harnack's inequality, we infer that  
\begin{equation}\label{nesbyo}
u\ge c_0M\rho \quad \hbox{in}\  B_{\frac{3\rho}{4}}(x_0)\subset  D,
\end{equation}
for some~$c_0>0$. Therefore, setting also
$$ \Sigma_\rho:=\partial B_\rho(x)\cap B_{\frac{3\rho}4}(x_0),$$
we conclude that
$$\fint_{\partial B_\rho(x)} u^+ \geq
c_1\fint_{\Sigma_\rho}u^+\geq c_0c_1M\rho,$$
where~$c_1>0$ depends only on the dimension $n$.

{F}rom this and~\eqref{wahlle}, we obtain that
\begin{equation}\label{nesbyo-1}
\fint_{\partial B_\rho(x)}u^-\geq  \fint_{\partial B_\rho(x)}u^+-C\rho
\geq (c_0c_1M-C)\rho>\frac M2\rho
\end{equation} 
if $M$ is large enough.

\smallskip 

Let $y\in \partial B_{\frac\rho2}(x_0)\cap[x,x_0]$ 
be the mid-point of the segment $[x,x_0]$. 
Then, by construction, we have that 
\begin{equation}\label{eifgreggr}
u^+\geq c_0 M\rho\quad {\mbox{ in }} B_{\frac\rho 4}(y),\end{equation} 
where $c_0$ is the constant in \eqref{nesbyo}. 

For our next computation, it is convenient  to switch to 
polar coordinates $(r, \sigma)$ centered at~$x$. 
Let $E_\rho$ be the set of $\sigma\in S^{n-1}$ such that~$u(x+\rho\sigma)<0$.
Let also~$I_\sigma$ be the ray that connects~$y$ and~$x+\rho\sigma$.
In what follows, we parameterize~$I_\sigma$ in arc-lenght by 
the parameter~$r\ge0$, with~$r=0$ corresponding to the point~$y$. 
The function~$u$ evaluated at the point of~$I_\sigma$ parameterized by~$r$
will be denoted by~$u(r)$. 

In this notation, formula~\eqref{eifgreggr} says that~$u^-(r)=0$ for any~$r
\in\left(0,\frac{\rho}{4}\right)$, and so
$$ \nabla u^-(r)=0 \quad {\mbox{ for any }}r
\in\left(0,\frac{\rho}{4}\right).$$
Then, recalling \eqref{nesbyo-1}, we have that
\begin{equation}\begin{split}\label{mek0}
&\frac M2 \leq \frac1{\rho}\fint_{\partial B_\rho(x)}u^-
=\frac{C}{\rho} \,\int_{E_\rho}
u^-(x+\rho\sigma)\, d{\mathcal{H}}^{n-1}(\sigma) =
\frac{C}{\rho}\, \int_{E_\rho}\int_{I_\sigma\cap B_\rho(x)} 
D_ru^-(r)\,dr\, d{\mathcal{H}}^{n-1}(\sigma)\\
&\qquad \le \frac{C}{\rho}\,\big(\rho \H^{n-1}(E_\rho)\big)^{\frac12} \,
\left( \int_{E_\rho}\int_{I_\sigma\cap B_\rho(x)} 
|D_ru^-(r)|^2\,dr \,d{\mathcal{H}}^{n-1}(\sigma)\right)^{\frac12}\\
&\qquad \le \frac{C}{\rho}\,\big(\rho \H^{n-1}(E_\rho)\big)^{\frac12} \,
\left( \int_{S^{n-1}}\int_{I_\sigma\cap \{r\in[\rho/4,\,2\rho]\}} 
|D_ru^-(r)|^2\,dr \,d{\mathcal{H}}^{n-1}(\sigma)\right)^{\frac12}\\
&\qquad \le \frac{C}{\rho}\,\big(\rho \H^{n-1}(E_\rho)\big)^{\frac12} \,
\left( \int_{B_{2\rho}(y)\setminus B_{\rho/4}(y)} \frac{
|\nabla u^-(z)|^2}{|z-y|^{n-1}}
\,dz\right)^{\frac12}\\
&\qquad \le \frac{C}{\rho}\,\big(\rho \H^{n-1}(E_\rho)\big)^{\frac12} \,
\left( \frac1\rho\,\int_{B_{2\rho}(y)\setminus B_{\rho/4}(y)} \frac{
|\nabla u^-(z)|^2}{|z-y|^{n-2}}
\,dz\right)^{\frac12},\end{split}\end{equation}
up to renaming~$C>0$ from line to line, 
where the H\"older's inequality was also used.

Now we observe that, if~$z\in B_{2\rho}(y)\setminus B_{\rho/4}(y)$,
we have that
$$ |z-x|\le |z-y|+|x-y|\le 3\rho$$
and so
$$ |z-y|\ge \frac{\rho}4\ge \frac{|z-x|}{12}.$$
Thus, renaming constants in~\eqref{mek0}, we obtain that
\begin{equation}\begin{split}\label{mek}
\frac M2 \,&
\le \frac{C}{\rho}\,\big(\rho \H^{n-1}(E_\rho)\big)^{\frac12} \,
\left( \frac1\rho\,\int_{B_{3\rho}(x)} \frac{
|\nabla u^-(z)|^2}{|z-x|^{n-2}}
\,dz\right)^{\frac12} \\ &=
\frac{C}{\rho}\,\big( \H^{n-1}(E_\rho)\big)^{\frac12} \,
\left( \int_{B_{3\rho}(x)} \frac{
|\nabla u^-(z)|^2}{|z-x|^{n-2}}
\,dz\right)^{\frac12}
.\end{split}\end{equation}
In order to estimate the integral average of $u^+$, we use~\eqref{eifgreggr}
and we obtain that (up to renaming constants)
\begin{eqnarray*}
&& M \rho^n\leq C\int_{\p B_{\frac{\rho}4}(y)} u^+= C\rho^{n-1} \int_{\partial B_1} u^+\left( y+\frac\rho4\omega\right)\,
d{\mathcal{H}}^{n-1}(\omega)\\&&\qquad\qquad
=
C\rho^{n-1} \int_{\partial B_1} \left[u^+\left( y+\frac\rho4\omega\right)-u^+(x)\right]\,
d{\mathcal{H}}^{n-1}(\omega)\\
&&\qquad\qquad\le C
\rho^{n-1} \int_{\partial B_1} \left[
\int_0^1
\left|\nabla u^+\left( x+r\left(y-x+\frac\rho4\omega\right)\right)\cdot\left(y-x+\frac\rho4\omega\right)
\right|
\,dr\right]\,
d{\mathcal{H}}^{n-1}(\omega)\\
&&\qquad\qquad\le C
\rho^{n} \int_{\partial B_1} \left[
\int_0^1
\left|\nabla u^+\left( x+r\left(y-x+\frac\rho4\omega\right)\right)
\right|
\,dr\right]\,
d{\mathcal{H}}^{n-1}(\omega)\\
&&\qquad\qquad\le C
\rho^{n-1} \int_{\partial B_1} \left[
\int_{\rho/4}^{2\rho}
\left|\nabla u^+\left( x+\bar r\bar\omega\right)
\right|
\,d\bar r\right]\,
d{\mathcal{H}}^{n-1}(\bar \omega)\\
&&\qquad\qquad= C
\rho^{n-1} \int_{\partial B_1} \left[
\int_{\rho/4}^{2\rho} \frac{{\bar r^{n-1}}\;
\left|\nabla u^+\left( x+\bar r\bar\omega\right)
\right|}{\bar r^{n-1}}
\,d\bar r\right]\,
d{\mathcal{H}}^{n-1}(\bar \omega)\\
&&\qquad\qquad= C
\rho^{n-1} \int_{B_{2\rho}(x)\setminus B_{\rho/4}(x)}
\frac{
\left|\nabla u^+\left( z\right)
\right|
}{|z-x|^{n-1}}\,dz\\
&&\qquad\qquad\le C
\rho^{n-1+\frac{n}2} \left(\int_{B_{2\rho}(x)\setminus B_{\rho/4}(x)}
\frac{
\left|\nabla u^+\left( z\right)
\right|^2
}{|z-x|^{2(n-1)}}\,dz\right)^{\frac12}
\\
&&\qquad\qquad\le C
\rho^{n-1} \left(\int_{B_{2\rho}(x)\setminus B_{\rho/4}(x)}
\frac{
\left|\nabla u^+\left( z\right)
\right|^2
}{|z-x|^{n-2}}\,dz\right)^{\frac12}.
\end{eqnarray*}
Combining this with~\eqref{mek}, and renaming constants, we conclude that
\begin{eqnarray*}
M^2 &\le& \frac{C}{\rho^2}\,\big( \H^{n-1}(E_\rho)\big)^{\frac12} \,
\left( \int_{B_{3\rho}(x)} \frac{
|\nabla u^-(z)|^2}{|z-x|^{n-2}}
\,dz\right)^{\frac12}
\,
\left(\int_{B_{2\rho}(x)\setminus B_{\rho/4}(x)}
\frac{
\left|\nabla u^+\left( z\right)
\right|^2
}{|z-x|^{n-2}}\,dz\right)^{\frac12}\\ &\le&
\frac{C}{\rho}\,
\left( \int_{B_{3\rho}(x)} \frac{
|\nabla u^-(z)|^2}{|z-x|^{n-2}}
\,dz\right)^{\frac12}
\,
\left(\int_{B_{3\rho}(x)}
\frac{
\left|\nabla u^+\left( z\right)
\right|^2
}{|z-x|^{n-2}}\,dz\right)^{\frac12}\\&=&
C\,
\big( \Phi(\rho,u^+,u^-)\big)^{\frac12}.
\end{eqnarray*}
Hence, from the monotonicity formula and~\eqref{5.27}, we get that
$$ M^2\le C\,
\left(1+\int_{B_1(x)}(u^+)^2+ \int_{B_1(x)}(u^-)^2\right)^{\frac12},
$$
which bounds~$M$, as desired.
\end{proof}

\section{Free boundary condition}\label{sec:FB-cond}

In this section, we will assume that the function~$Q$ introduced in~\eqref{Q}
is continuous.

Next result shows that, at points~$p$ of the free boundary,
the following condition holds true in the sense of distributions:
\begin{eqnarray*}&& \big(\partial_\nu^+ u(p)\big)^2 - \big(\partial_\nu^- u(p)\big)^2\\&&\quad= 
\left[\lambda_2\partial_{r_2}\Phi
\left(\lambda_1
\int_\Om Q
\,\chi_{ \{ u<0\} },\;
\lambda_2
\int_\Om Q
\,\chi_{ \{ u>0\} }\right)\,
-\,
\lambda_1\partial_{r_1}\Phi
\left(\lambda_1
\int_\Om Q
\,\chi_{ \{ u<0\} },\;
\lambda_2
\int_\Om Q
\,\chi_{ \{ u>0\} }\right)\,
\right]\,Q(p)
,\end{eqnarray*}
where~$\nu$ is the normal vector exterior to~$\partial\{u>0\}$
(and thus pointing towards~$\{u\le0\}$) and we set
\begin{equation}\label{no00n}
\partial_{\nu}^+ u (x):=\lim_{t\to0}\frac{u(x-t\nu)-u(x)}{t}\;{\mbox{ and }}\;
\partial_{\nu}^- u (x):=\lim_{t\to0}\frac{u(x+t\nu)-u(x)}{t}
.\end{equation}
More precisely, we have that:

\begin{lem}\label{FBCOND}
Let~$u$
be a minimizer of $J$
as in~\eqref{EN:FUNCT} 
and suppose
that~$Q\in W^{1,1}(\Omega)$. 
Assume also that
\begin{equation}\label{ME0}
{\mbox{the set~$\{u=0\}\cap\Omega$ has zero measure.}}
\end{equation}
Then
$$\lim_{\e\searrow0}\left\{
\int_{\partial { \{ u>\e\} }\cap\Omega}  {\mathcal{I}}^{\e,+}(u,x)\,
V(x)\cdot\nu^{\e,+}(x)\,d{\mathcal{H}}^{n-1}(x)
+
\int_{\partial { \{ u<-\e\} }\cap\Omega} {\mathcal{I}}^{\e,-}(u,x)
\,V(x)\cdot\nu^{\e,-}(x)\,d{\mathcal{H}}^{n-1}(x)\right\}
=0$$
for any vector field~$V\in C^\infty_0(\Omega,\R^n)$,
where we have denoted
by~$\nu^{\e,+}$ and~$\nu^{\e,-}$ the exterior normals of~$\{u>\e\}$
and~$\{u<-\e\}$, respectively, and
\begin{equation}\label{DEF:I}
\begin{split}
& {\mathcal{I}}^{\e,+}(u,x):=
|\partial_{\nu}^+u(x)|^2 
-\lambda_2\partial_{r_2}\Phi
\left(\lambda_1
\int_\Om Q(\xi)
\,\chi_{ \{ u<-\e\} }(\xi)\,d\xi,\;
\lambda_2
\int_\Om Q(\xi)
\,\chi_{ \{ u>\e\} }(\xi)\,d\xi\right)\,
Q(x)\\ {\mbox{and }}\;&
{\mathcal{I}}^{\e,-}(u,x):=
|\partial_{\nu}^-u(x)|^2 
-
\lambda_1\partial_{r_1}\Phi
\left(\lambda_1
\int_\Om Q(\xi)
\,\chi_{ \{ u<-\e\} }(\xi)\,d\xi,\;
\lambda_2
\int_\Om Q(\xi)
\,\chi_{ \{ u>\e\} }(\xi)\,d\xi\right)
\,Q(x).
\end{split}\end{equation}
\end{lem}

\begin{proof} The argument is a (not completely straightforward)
modification of the classical domain variations in Theorem~2.5
in~\cite{AC} and in Theorem~2.4 of~\cite{ACF}. 
We provide full details for the facility of the reader.
For small~$t\in\R$, we consider the ODE flow~$y=y(t;x)$ given
by the Cauchy problem
\begin{equation*}
\left\{
\begin{matrix}
\partial_t y(t;x) = V(y(t;x)),\\
y(0;x)= x.
\end{matrix}
\right.\end{equation*}
The map~$\R^n\ni x\mapsto y(t;x)$
is invertible for small~$t$, i.e. we can consider
the inverse diffeomorphism~$x(t;y)$ and we define
$$ u_t (y):= u( x(t;y) ).$$
We remark that, in light of~\eqref{ME0},
\begin{equation}\label{ME0:2}
{\mbox{the set~$\{u_t=0\}\cap\Omega$ has zero measure.}}
\end{equation}
Given~$\e>0$, we define~${E^{\e,+}}:=\{u>\e\}\cap\Om$,
${E^{\e,-}}:=\{u<\e\}\cap\Om$
and~${E^{\e,\pm}_t}:=y(t;{E^{\e,\pm}})$. 
Notice that
\begin{equation}\label{PDI}
\{u_t>\e\}\cap\Om = {E^{\e,+}_t}\;{\mbox{ and }}\;
\{u_t<\e\}\cap\Om = {E^{\e,-}_t}.
\end{equation}
One can check (see e.g. formulas~(4.5), (4.13) and~(4.22) in~\cite{DKV})
that
\begin{eqnarray}
\label{PO:A1}&& y(t;\Om)=\Om,\\
\label{PO:A2}&& \det D_x y(t;x) = 1+t\,{\rm div} V(x) +o(t)\\
{\mbox{and }}
\label{PO:A3}&& 
\int_\Omega |\nabla u(x)|^2\,dx-\int_\Omega |\nabla u_t(y)|^2\,dy= \lim_{\e\searrow0}
t\,\int_{\partial {E^{\e,+}}\cap\Omega}
|\partial_{\nu}^+u(y)|^2 \,V(y)\cdot\nu^{\e,+}(y)\,d{\mathcal{H}}^{n-1}(y)
\\ \nonumber&&\qquad\qquad+
t\,\int_{\partial {E^{\e,-}}\cap\Omega}
|\partial_{\nu}^-u(y)|^2 \,V(y)\cdot\nu^{\e,-}(y)\,d{\mathcal{H}}^{n-1}(y)+
o(t).\end{eqnarray}
By approximating~$Q$ by a sequence of
smooth functions in~$W^{1,1}(\Omega)$
and a.e. in~$\Omega$, we deduce from~\eqref{PO:A1} and~\eqref{PO:A2} that
\begin{equation*}\begin{split}
&
\int_\Om Q(y)\,\chi_{{E^{\e,\pm}_t}}(y)\,dy=
\int_\Om Q\big(y(t;x)\big)
\,\chi_{{E^{\e,\pm}}}(x)\,\big( 1+t\,{\rm div} V(x) +o(t)\big)\,dx
\\&\qquad=
\int_\Om Q(x)\,\chi_{{E^{\e,\pm}}}(x)\,\big( 1+t\,{\rm div} V(x) +o(t)\big)\,dx\\&\qquad\qquad+
\int_{\Omega\times(0,1)} \nabla Q\big( x+\tau(y(t;x)-x)\big)
\cdot(y(t;x)-x)
\,\chi_{{E^{\e,\pm}}}(x)\,\big( 1+t\,{\rm div} V(x) +o(t)\big)
\,dx\,d\tau.\end{split}\end{equation*}
Therefore
\begin{equation}\label{minor:2}\begin{split}
&
\int_\Om Q(y)\,\chi_{{E^{\e,\pm}_t}}(y)\,dy\\&\qquad=
\int_\Om Q(x)\,\chi_{{E^{\e,\pm}}}(x)\,dx
+t\,\int_\Om {\rm div}\,(Q(x)V(x))\,\chi_{{E^{\e,\pm}}}(x)\,dx
-t\int_{ {E^{\e,\pm}} }\nabla Q(x)\cdot V(x)\,dx
\\&\qquad\qquad+t
\int_{{E^{\e,\pm}}\times(0,1)} \nabla Q\big( x+\tau t V(x)+o(t)\big)\cdot V(x)
\,dx\,d\tau
+o(t).
\end{split}\end{equation}
Furthermore, changing variables again, we see that
\begin{equation}\label{minor:3}\begin{split}&
\int_{{E^{\e,\pm}}\times(0,1)} \nabla Q\big( x+\tau t V(x)+o(t)\big)\cdot V(x)
\,dx\,d\tau\\&\qquad=
\int_{{\tilde E^{\e,\pm}_t}\times(0,1)} \nabla Q(y)\cdot V(
y-\tau t V(y)+o(t)) \,(1-t\,{\rm div}V(y)+o(t))
\,dy\,d\tau
\\ &\qquad=\int_{{\tilde E^{\e,\pm}_t}\times(0,1)} \nabla Q(y)\cdot V(
y) 
\,dy\,d\tau+o(1)
\\ &\qquad=\int_{E^{\e,\pm}} \nabla Q(y)\cdot V(
y) 
\,dy+o(1),
\end{split}\end{equation}
for a suitable set~${\tilde E^{\e,\pm}_t}$ which is close, for small~$t$,
to~$E^{\e,\pm}$.
{F}rom~\eqref{minor:2} and~\eqref{minor:3}, we have that
\begin{equation*}
 \int_\Om Q(y)\,\chi_{{E^{\e,\pm}_t}}(y)\,dy=
\int_\Om \big( Q(x)+ t\,{\rm div} \big(Q(x)V(x)\big) \big)
\,\chi_{{E^{\e,\pm}}}(x)\,dx+o(t).
\end{equation*}
Consequently, we can linearize~$\Phi$ and obtain
\begin{equation}\label{FoR}
\begin{split}
& \Phi\left(\lambda_1\int_\Om Q(y)\,\chi_{{E^{\e,-}_t}}(y)\,dy,\;
\lambda_2\int_\Om Q(y)\,\chi_{{E^{\e,+}_t}}(y)\,dy
\right)\\ 
=\;&
\Phi\left(\lambda_1
\int_\Om Q(x)
\,\chi_{{E^{\e,-}}}(x)\,dx, \;\lambda_2
\int_\Om Q(x)
\,\chi_{{E^{\e,+}}}(x)\,dx\right)\\
&\quad
+t\lambda_1\partial_{r_1}\Phi
\left(\lambda_1
\int_\Om Q(x)
\,\chi_{{E^{\e,-}}}(x)\,dx,\;
\lambda_2
\int_\Om Q(x)
\,\chi_{{E^{\e,+}}}(x)\,dx
\right)
\int_\Om {\rm div} \big(Q(x)V(x)\big)
\,\chi_{{E^{\e,-}}}(x)\,dx \\
&\quad
+t\lambda_2\partial_{r_2}\Phi
\left(\lambda_1
\int_\Om Q(x)
\,\chi_{{E^{\e,-}}}(x)\,dx,\;
\lambda_2
\int_\Om Q(x)
\,\chi_{{E^{\e,+}}}(x)\,dx
\right)
\int_\Om {\rm div} \big(Q(x)V(x)\big)
\,\chi_{{E^{\e,+}}}(x)\,dx 
+o(t)\\
=\;&
\Phi\left(\lambda_1
\int_\Om Q(x)
\,\chi_{{E^{\e,-}}}(x)\,dx, \;\lambda_2
\int_\Om Q(x)
\,\chi_{{E^{\e,+}}}(x)\,dx\right)
\\ &\quad+t\lambda_1\partial_{r_1}\Phi
\left(\lambda_1
\int_\Om Q(x)
\,\chi_{{E^{\e,-}}}(x)\,dx,\;
\lambda_2
\int_\Om Q(x)
\,\chi_{{E^{\e,+}}}(x)\,dx\right)
\int_{\partial {E^{\e,-}}\cap\Om }Q(x)V(x)\cdot\nu^{\e,-}(x)\,d{\mathcal{H}}^{n-1}(x)
\\ &\quad+t\lambda_2\partial_{r_2}\Phi
\left(\lambda_1
\int_\Om Q(x)
\,\chi_{{E^{\e,-}}}(x)\,dx,\;
\lambda_2
\int_\Om Q(x)
\,\chi_{{E^{\e,+}}}(x)\,dx\right)
\int_{\partial {E^{\e,+}}\cap\Om }Q(x)V(x)\cdot\nu^{\e,+}(x)\,d{\mathcal{H}}^{n-1}(x)
\\ &\qquad+o(t).
\end{split}
\end{equation}
Moreover, by inspection and recalling~\eqref{PDI}, one sees that
$$ \lim_{\e\searrow0} \chi_{E^{\e,+}_t}= \lim_{\e\searrow0} \chi_{\{ u_t>\e\}}=
\chi_{\{ u_t>0\}} ,$$
for any small~$t\ge0$.
Similarly, and using~\eqref{ME0:2},
$$ \lim_{\e\searrow0} \chi_{E^{\e,-}_t}=
\chi_{\{ u_t<0\}} =\chi_{\{ u_t\le0\}}$$
a.e. in~$\Omega$. As a consequence, by the Dominated Convergence Theorem,
$$ {\mathcal{M}}_1(u_t) = \lim_{\e\searrow0}
\lambda_1
\int_\Om Q(x)
\,\chi_{{E^{\e,-}_t}}(x)\,dx
\;\;{\mbox{ and }}\;\;
{\mathcal{M}}_2(u_t) = \lim_{\e\searrow0}
\lambda_2
\int_\Om Q(x)
\,\chi_{{E^{\e,+}_t}}(x)\,dx,$$
for any small~$t\ge0$. So, we can take the limit with respect to~$\e$ in formula~\eqref{FoR}
and obtain that
\begin{eqnarray*}
&& \Phi\big({\mathcal{M}}_1(u_t),\;{\mathcal{M}}_2(u_t)\big)-
\Phi\big({\mathcal{M}}_1(u),\;{\mathcal{M}}_2(u)\big)\\
&=&\lim_{\e\searrow0}
t\lambda_1\partial_{r_1}\Phi
\left(\lambda_1
\int_\Om Q(x)
\,\chi_{{E^{\e,-}}}(x)\,dx,\;
\lambda_2
\int_\Om Q(x)
\,\chi_{{E^{\e,+}}}(x)\,dx\right)
\int_{\partial {E^{\e,-}}\cap\Om }Q(x)V(x)\cdot\nu^{\e,-}(x)\,d{\mathcal{H}}^{n-1}(x)
\\ &&\quad+t\lambda_2\partial_{r_2}\Phi
\left(\lambda_1
\int_\Om Q(x)
\,\chi_{{E^{\e,-}}}(x)\,dx,\;
\lambda_2
\int_\Om Q(x)
\,\chi_{{E^{\e,+}}}(x)\,dx\right)
\int_{\partial {E^{\e,+}}\cap\Om }Q(x)V(x)\cdot\nu^{\e,+}(x)\,d{\mathcal{H}}^{n-1}(x)
\\ &&\qquad+o(t).
\end{eqnarray*}

{F}rom this and
\eqref{PO:A3}, we have that
\begin{eqnarray*}&&
J[u]-J[u_t]\\
&=&
\int_\Omega |\nabla u(x)|^2\,dx-\int_\Omega |\nabla u_t(y)|^2\,dy
\\ &&\qquad+
\Phi\big({\mathcal{M}}_1(u),\;{\mathcal{M}}_2(u)\big)-
\Phi\big({\mathcal{M}}_1(u_t),\;{\mathcal{M}}_2(u_t)\big)
\\
&=& \lim_{\e\searrow0}
t\,\int_{\partial {E^{\e,+}}\cap\Omega}
|\partial_{\nu}^+u(y)|^2 \,V(y)\cdot\nu^{\e,+}(y)\,d{\mathcal{H}}^{n-1}(y)
+
t\,\int_{\partial {E^{\e,-}}\cap\Omega}
|\partial_{\nu}^-u(y)|^2 \,V(y)\cdot\nu^{\e,-}(y)\,d{\mathcal{H}}^{n-1}(y)
\\ &&-
t\lambda_1\partial_{r_1}\Phi
\left(\lambda_1
\int_\Om Q(x)
\,\chi_{{E^{\e,-}}}(x)\,dx,\;
\lambda_2
\int_\Om Q(x)
\,\chi_{{E^{\e,+}}}(x)\,dx\right)
\int_{\partial {E^{\e,-}}\cap\Om }Q(x)V(x)\cdot\nu^{\e,-}(x)\,d{\mathcal{H}}^{n-1}(x)
\\ &&\quad-t\lambda_2\partial_{r_2}\Phi
\left(\lambda_1
\int_\Om Q(x)
\,\chi_{{E^{\e,-}}}(x)\,dx,\;
\lambda_2
\int_\Om Q(x)
\,\chi_{{E^{\e,+}}}(x)\,dx\right)
\int_{\partial {E^{\e,+}}\cap\Om }Q(x)V(x)\cdot\nu^{\e,+}(x)\,d{\mathcal{H}}^{n-1}(x)
\\ &&\qquad+o(t).
\end{eqnarray*}
Dividing by~$t$ and then letting $t\to 0$, we obtain the desired result.
\end{proof}

We observe that when~$\Phi(r_1,r_2):=r_1+r_2$, then \eqref{DEF:I}
reduces to~${\mathcal{I}}^{\e,+}(u,x):=
|\partial_{\nu}^+u(x)|^2 
-\lambda_2\,Q(x)$ and~$
{\mathcal{I}}^{\e,-}(u,x):=
|\partial_{\nu}^-u(x)|^2 
-\lambda_1\,Q(x)$. Hence, in this particular case,
our Lemma~\ref{FBCOND} boils down to Theorem~2.4 in~\cite{ACF}.
If also~$\lambda_1:=0$, then 
Lemma~\ref{FBCOND} boils down to Theorem~2.5 in~\cite{AC}.

\medskip 
Next we show that $\fb{u}$ contains $\partial\{u<0\}$ under the condition \eqref{LAMBDA12}.
Thus one has sharp separation of phases.
\begin{lem}
Let $u$ be a minimizer in~$\Omega$ of the functional in~\eqref{EN:FUNCT}. 
Then $\partial\{u<0\}\setminus \partial\{u>0\}=\varnothing$.
\end{lem}

\begin{proof}
Let~$E:=\partial\{u<0\} \setminus \partial\{u>0\}$. We want to show
that~$E$ is empty
and suppose by contradiction
that~$E\neq \varnothing$. Then, there exist~$p\in\Omega$ and~$r>0$
such that~$u\le 0$ in $B_r(p)$, with~$u(p)=0$ and
${\mathcal{L}}^n\big(B_r(p)\cap \{u<0\}\big)>0$. 
Then, we use that~$u$ is subharmonic, in view of Lemma~\ref{LEMMA2},
and we obtain that 
$$ u(p)\le\int_{B_r(p)} u<0,$$
which is a contradiction.
\end{proof}

\section{Nondegeneracy of minimizers}\label{sec:nondeg} 

One of the fundamental properties of 
the minimizers is a linear lower bound. In other words, the minimzers grow
at least linearly  away from the free boundary. This is the content of
Theorem~\ref{prop-nondeg}, that we now prove:

\begin{proof}[Proof of Theorem~\ref{prop-nondeg}]
Let $\kappa\in(0,1)$. Without loss of generality we assume that~$x_0=0$. 

Notice that the critical points of a non-constant harmonic function 
have Hausdorff dimension less than $n-2$. From Sard's theorem it follows that 
the one dimensional Lebesgue measure of the critical values of $u$ is zero. 
Consequently, $\partial\{u>\e\}$ is a regular surface for a.e. $\e>0$. 
In particular, one can choose~$\e>0$ small enough 
to ensure that~$B_{\kappa r}\cap\{u>\e\}\not=\emptyset$.

Now take any such small~$\e>0$, and consider the problem 
\begin{equation}\label{constraints}
\left\{
\begin{array}{lll}
v_\e= u &\text{on}\ \partial B_r,\\
v_\e=u &\text{in}\ B_r\cap\Big( \{u<\e\}\cup \big(\{u>\e\}\setminus \Om^+_0\big)\Big),\\
v_\e=\e &\text{in} \ B_{\kappa r}\cap \{u>\e\}\cap \Om^+_0,\\
\Delta v_\e=0 &\text{in} \ D_\e^+,
\end{array}
\right.
\end{equation}
where 
\begin{equation}\label{D:e}
D_\e^+=(B_r\setminus B_{\kappa r})\cap \{u>\e\}\cap \Om^+_0.
\end{equation}
Observe that~$v_\e$ can be obtained by minimizing the Dirichlet integral 
over $B_r$ subject to the constraints in~\eqref{constraints}. 
In fact, the function
\begin{equation*}
w_\e=\left\{
\begin{array}{lll}
u &\text{in}\ \{u<\e\}\\
\e  &\text{in}\ \{u>\e\}\cap \Om^+_0\cap B_{\kappa r}\\
u & \text{elsewhere}
\end{array}
\right.
\end{equation*}
satisfies the boundary constraints in~\eqref{constraints}, 
and hence 
\begin{equation}\label{po987}
\int_{B_r}|\nabla v_\e|^2\leq \int_{B_r}|\nabla w_\e|^2\leq C,
\end{equation}
for some tame constant $C>0$ independent of $\e$. 
Also,~$v_\e$ is continuous at~$\{u=\e\}\cap (B_r\setminus B_{\kappa r})$.
We claim that
\begin{equation}\label{ST:ve}
v_\e\le u \;{\mbox{ in }}\overline{D_\e^+}.\end{equation}
Indeed, by inspection we see that~$v_\e\le u$ on~$\partial D_\e^+$.
Moreover, $D_\e^+\subseteq\{ u>\e\}$ and so~$u$ is harmonic there,
thanks to Corollary~\ref{COR:5}.
Hence, we obtain~\eqref{ST:ve} from
the comparison principle.

Now,
formulas~\eqref{po987} and~\eqref{ST:ve}
imply that~$v_\e\to v$ weakly in $W^{1, 2}(B_r)$, as~$\epsilon\to0$, 
and~$v\le u$. Furthermore,~$v$ is continuous in~$B_r$ and solves 
\begin{equation*}
\left\{
\begin{array}{lll}
v= u &\text{on}\ \partial B_r,\\
v=u &\text{in}\ B_r\cap\Big( \{u\le 0\}\cup \big(\{u>0\}\setminus \Om^+_0\big)\Big),\\
v=0 &\text{in} \ B_{\kappa r}\cap \Omega_0^+,\\
\Delta v=0 &\text{in} \ D^+,
\end{array}
\right.
\end{equation*}
where
\begin{equation}\label{def D}
D^+:=(B_r\setminus B_{\kappa r})\cap \Omega_0^+. \end{equation}
The former follows from a 
customary approximation argument as on page 437 of~\cite{ACF}, and hence omitted here. 

Now we extend~$v$ to be equal to~$u$ in~$\Om\setminus B_r$ and we 
compare~$J[u]$ with~$J[v]$ in~$\Omega$. Accordingly, the minimality of~$u$ gives that 
\begin{equation*}
\int_{\Om}|\nabla u|^2+ \Phi_0 \left(\lambda_2\int_{\Om}Q\I{u}\right)
\le  \int_{\Om} |\nabla v|^2+ \Phi_0 \left(\lambda_2\int_\Om Q\I{v}\right).
\end{equation*}
This implies that 
\begin{equation}\label{zzzz}
\int_{B_r}|\nabla u|^2-\int_{B_r} |\nabla v|^2 = 
\int_{\Omega}|\nabla u|^2-\int_{\Omega} |\nabla v|^2 
\le \Phi_0 \left(\lambda_2\int_\Om Q\I{v}\right)- \Phi_0 \left(\lambda_2\int_{\Om}Q\I{u}\right).
\end{equation}
On the other hand, recalling that~$v=u$ in~$B_r\setminus\Omega_0^+$ 
and~$v=0$ in~$B_{\kappa r}\cap\Omega_0^+$, we see that 
\begin{equation}\begin{split}
& \int_{B_r}|\nabla u|^2-|\nabla v|^2 = \int_{B_{r}\cap \Omega_0^+}|\nabla u|^2-|\nabla v|^2 \\
&\qquad\qquad =\int_{D^+}|\nabla u|^2-|\nabla v|^2 + \int_{B_{\kappa r}\cap\Omega_0^+}|\nabla u|^2, 
\end{split}\end{equation}
where~\eqref{def D} was also used. 
This and~\eqref{zzzz} give that 
\begin{equation}\label{ujergt234}
\int_{D^+}|\nabla u|^2-|\nabla v|^2 + \int_{B_{\kappa r}\cap\Omega_0^+}|\nabla u|^2 \le 
 \Phi_0 \left(\lambda_2\int_\Om Q\I{v}\right)- \Phi_0 \left(\lambda_2\int_{\Om}Q\I{u}\right).
\end{equation}

Now we write 
\begin{equation}\label{ywre543}\begin{split}
& \Phi_0 \left(\lambda_2\int_\Om Q\I{u}\right)- \Phi_0 \left(\lambda_2\int_{\Om}Q\I{v}\right) \\
=&\, \Phi_0 \left(\lambda_2\int_{\Om\setminus B_{\kappa r}} Q\I{u} +
\lambda_2\int_{B_{\kappa r}}Q\I{u}\right)- 
\Phi_0 \left(\lambda_2\int_{\Om\setminus B_{\kappa r}}Q\I{v} +\lambda_2 \int_{B_{\kappa r}} Q\I{v} \right).
\end{split}\end{equation}
Notice that $v=0$ in~$B_{\kappa r}\cap \Omega_0^+$ and~$v=u$ in~$B_{\kappa r}\setminus \Omega_0^+$, hence
\begin{equation}\label{ywre543-1}
\int_{B_{\kappa r}} Q\I{v} = \int_{B_{\kappa r}\setminus\Omega_0^+} Q\I{u}.
\end{equation}
Moreover, $v=u$ in~$\Omega\setminus B_r$ and in~$(B_r\setminus B_{\kappa r})\setminus \Omega_0^+$.
Also, in~$(B_r\setminus B_{\kappa r})\cap\Omega_0^+=D^+$, we have
that~$u>0$, by definition of~$\Omega_0^+$, 
and therefore 
$$ \I{u} \ge \I{v} \quad {\mbox{ in }} D^+.$$
This implies that 
\begin{equation}\label{ywre543-2}
\int_{\Om\setminus B_{\kappa r}}Q\I{v} \le \int_{\Om\setminus B_{\kappa r}}Q\I{u}.
\end{equation}
Plugging~\eqref{ywre543-1} and~\eqref{ywre543-2} into~\eqref{ywre543}, and using~\eqref{LAMBDA12},
we obtain that 
\begin{equation*}\begin{split}
& \Phi_0 \left(\lambda_2\int_\Om Q\I{u}\right)- \Phi_0 \left(\lambda_2
\int_{\Om}Q\I{v}\right) \\
\geq&\, \Phi_0 \left(\lambda_2\int_{\Om\setminus B_{\kappa r}} Q\I{u} + 
\lambda_2\int_{B_{\kappa r}}Q\I{u}\right)-
\Phi_0 \left(\lambda_2\int_{\Om\setminus B_{\kappa r}}Q\I{u} + 
\lambda_2\int_{B_{\kappa r}\setminus\Omega_0^+} Q\I{u} \right).
\end{split}\end{equation*}
Therefore, from the Mean Value Theorem we get  
\begin{equation}\label{retnvjfhl0}\begin{split}
& \Phi_0 \left(\lambda_2\int_{\Om}Q\I{u}\right)-\Phi_0 \left(\lambda_2
\int_{\Om}Q\I{v}\right)\\&\qquad \ge \Theta\left( \lambda_2
\int_{B_{\kappa r}}Q\I{u}- \lambda_2\int_{B_{\kappa r}\setminus \Omega_0^+}Q\I{u}\right) 
=\Theta \lambda_2\int_{B_{\kappa r}\cap\Omega_0^+}Q\I{u},
\end{split}\end{equation}
where we recalled the notation in~\eqref{IPTH},
and~\eqref{JOHN} has been used to estimate the interval in the definition of~$\Theta$.

So, from~\eqref{ujergt234} and~\eqref{retnvjfhl0} we deduce that 
\begin{equation*}
\int_{D^+}|\nabla u|^2-|\nabla v|^2 + \int_{B_{\kappa r}\cap\Omega_0^+}|\nabla u|^2 
\le -\Theta\lambda_2\int_{B_{\kappa r}\cap\Omega_0^+}Q\I{u}.\end{equation*}
Using this, we obtain that 
\begin{equation}\label{final00}
\min\left\{1, \,\Theta\right\}\, 
\int_{B_{\kappa r}\cap\Omega_{0}^+}\left( |\nabla u|^2 +\lambda_2 Q\I{u}\right) 
\le \int_{D^+}|\nabla v|^2-|\nabla u|^2.\end{equation}

Now we observe that
\begin{eqnarray*}
&& \int_{D^+}|\nabla v|^2-|\nabla u|^2 = \int_{D^+}\left(\nabla v-\nabla u\right)\cdot \left(
\nabla u-\nabla v+2\nabla v\right) \\
&&\qquad = -\int_{D^+}|\nabla u-\nabla v|^2 + 2\int_{D^+}\nabla v\cdot\left(\nabla v-\nabla u\right) 
\le 2\int_{D^+}\nabla v\cdot\left(\nabla v-\nabla u\right).
\end{eqnarray*}
Plugging this into~\eqref{final00}, recalling that~$v_\epsilon$ is a solution of~\eqref{constraints}
and using the Divergence Theorem, we obtain 
\begin{equation}\begin{split}\label{zzzz1}
& \min\left\{1, \,\Theta\right\}\, 
\int_{B_{\kappa r}\cap\Omega_{0}^+}\left( |\nabla u|^2 +\lambda_2
Q\I{u}\right) \\&\qquad \le 
2\int_{D^+}\nabla v\cdot (\nabla v-\nabla u)\\&\qquad \le
2\liminf_{\e\to 0}\int_{D_\e^+}\nabla v_\epsilon \cdot (\nabla v_\epsilon -\nabla u)
\\&\qquad = 2\liminf_{\e\to 0}\left[\int_{D_\e^+} {\rm{div}}\left(\nabla v_\epsilon(v_\epsilon-u)\right)
-\int_{D_\e^+}\Delta v_\epsilon (v_\epsilon-u)\right] 
\\&\qquad = 2\liminf_{\e\to 0} \int_{\partial D_\epsilon^+}
\frac{\partial v_\epsilon}{\partial\nu}\,(v_\epsilon-u)\\&\qquad  \le
2\liminf_{\e\to 0} \int_{\partial B_{\kappa r}\cap\{u>\epsilon\}\cap\Omega_0^+} 
\left|\frac{\partial v_\epsilon}{\partial\nu}\right|\,(u-\epsilon)
.\end{split}\end{equation}
Notice that we have performed an integration by parts, which actually needs a justification, 
since~$D_\epsilon^+$ has not smooth boundary. This can be done 
using an approximation of~$D_\e^+$ by domains whose boundaries are~$C^\infty$ curves, as in~\cite{ACF} 
(see in particular page 437 there).

Our next goal is to estimate the quantity
$$ 2\liminf_{\e\to 0} \int_{\partial B_{\kappa r}\cap\{u>\epsilon\}\cap\Omega_0^+} 
\left|\frac{\partial v_\epsilon}{\partial\nu}\right|\,(u-\epsilon).$$ 
For this, we set~$\kappa':=(\kappa+1)/2$
and we introduce the barrier $b$ as follows:
\begin{equation}\label{b:eq}
\left\{
\begin{array}{lll}
b=\e+\displaystyle\sup_{\partial B_{\kappa'r}\cap\Om_0^+}v_\e &\text{on}\ \partial B_{\kappa' r},\\
b= \e &\text{on }\ \partial B_{\kappa r},\\
\Delta b=0 &\text{in} \ B_{\kappa' r}\setminus B_{\kappa r}.
\end{array}
\right.
\end{equation}
Notice that~$b\ge\epsilon$ on~$\partial \left(B_{\kappa' r}\setminus B_{\kappa r}\right)$
and so, by comparison principle, we have that
\begin{equation}\label{VHA}
{\mbox{$b\ge \epsilon$
in~$\overline{B_{\kappa' r}}\setminus B_{\kappa r}$.}}\end{equation}
Recalling~\eqref{D:e}, we set
$$ D_{\e,*}^+:=D_\e^+\cap B_{\kappa' r}=
(B_{\kappa' r}\setminus B_{\kappa r})\cap \{u>\e\}\cap \Om^+_0$$
and we
claim that
\begin{equation}\label{THC}
{\mbox{$v_\e\leq b$ on~$\partial D^+_{\epsilon,*}$}}.\end{equation}
For this, we use the elementary formula, given sets~$A$ and $B$,
\begin{equation} \label{HIN}
\partial(A\cap B)\subseteq \Big( \partial A\cap\overline{B}\Big)
\cup \Big( \partial B\cap\overline{A}\Big).\end{equation}
This gives that
$$ \partial D^+_{\epsilon,*}\subseteq D_1\cup D_2,$$
with
\begin{eqnarray*}
&& D_1:=\big(\partial B_{\kappa r}\cup\partial B_{\kappa'r}\big)\cap
\{ x\in\overline{\Omega_0^+} {\mbox{ s.t. }} u(x)\ge\e\}\cap \overline{D_\e^+}
\\
{\mbox{and }}&& D_2:= \big(\overline{B_{\kappa'r}}\setminus B_{\kappa r}\big)
\cap \big(\partial \{ x\in{\Omega_0^+} {\mbox{ s.t. }} u(x)>\e\}\big)\cap \overline{D_\e^+}.
\end{eqnarray*}
Now, in light of~\eqref{b:eq} and~\eqref{constraints},
if~$x\in D_1$, we have that either~$x\in \partial B_{\kappa'r}
\cap\{u\ge\epsilon\}\cap \overline{\Om^+_0}$,
and then~$b(x)\ge \sup_{\partial B_{\kappa'r}\cap\Om_0^+}v_\e\ge v_\e(x)$,
or~$x\in \partial B_{\kappa r}\cap
\{u\ge\epsilon\}\cap \overline{\Om^+_0}$
and~$b(x)=\e=v_\e(x)$. Accordingly,
\begin{equation}\label{LA:A}
{\mbox{$b\geq v_\e$ in~$D_1$.}}
\end{equation}
On the other hand,
using again~\eqref{HIN}, one sees that
$$ \partial \{ x\in{\Omega_0^+} {\mbox{ s.t. }} u(x)>\e\}\;\subseteq\;
(\partial\Omega_0^+)\cup(\partial\{u>\e\})\;\subseteq\;
(\partial\Omega)\cup(\partial\{u>0\})\cup(\partial\{u>\e\})$$
and so
$$ D_2\subseteq \big(\overline{B_{\kappa'r}}\setminus B_{\kappa r}\big)\cap\Big(\{u=0\}
\cup\{u=\e\}\Big)\cap \overline{D_\e^+}.$$
As a consequence, if~$x\in D_2$, then~$x\in\overline{B_{\kappa'r}}
\setminus B_{\kappa r}$
and~$u(x)\le\e$. This and~\eqref{VHA} give that~$u\le b$ in~$D_2$.
Hence, in view of~\eqref{ST:ve}, we find that~$v_\e\le b$ in~$D_2$.
This, together with \eqref{LA:A}, proves~\eqref{THC}.

\smallskip

Now, by~\eqref{constraints}, \eqref{b:eq},
\eqref{THC} and the comparison principle, we conclude that~$
v_\e\leq b$ in~$D^+_{\epsilon,*}$.
Therefore, since~$v_\e=\e=b$ on~$\partial B_{\kappa r}\cap \{u>\e\}\cap \Om^+_0$,
we find that
\begin{equation}\label{D1}
{\mbox{
$|\nabla v_\e|\leq |\nabla b|$
on~$\partial B_{\kappa r}\cap \{u>\e\}\cap \Om^+_0.$}}\end{equation}

As a matter of fact, we can explicitly solve~$b$ in~\eqref{b:eq},
and we have that
$$ b(x)= \frac{\sup_{\partial B_{\kappa'r}\cap\Om_0^+} v_\e}{ 
\Psi(\kappa r)-\Psi(\kappa'r)}\,
\big(\Psi(\kappa'r)-\Psi(|x|)\big)+\e+
\sup_{\partial B_{\kappa'r}\cap\Om_0^+} v_\e,$$
where~$\Psi$ is the radially decreasing fundamental solution of the
Laplace operator in~$\R^n$ (up to normalizing constants, $\Psi(\rho)=\rho^{2-n}$
if~$n\ge3$ and~$\Psi(\rho)=-\log\rho$ if~$n=2$).

Consequently, recalling also~\eqref{ST:ve},
\begin{equation}\label{KA90AK}
\sup_{\partial B_{\kappa r}}|\nabla b| = C
\frac{\sup_{\partial B_{\kappa'r}\cap\Om_0^+} v_\e}{ \Psi(\kappa r)-\Psi(
\kappa'r)}\,|\nabla\Psi(\kappa r)|
\le C \frac{\sup_{ \partial B_{\kappa'r}\cap \Om^+_0}u}{r},\end{equation}
for some~$C>0$ possibly depending on~$\kappa$ and
different from step to step.

Now, we observe that, extending $u$ by zero outside~$\Omega_0^+$,
we obtain a 
nonnegative subharmonic function.
More precisely, if we set~$\tilde u:=u\chi_{\Omega_0^+}$,
given any nonnegative~$\phi\in C^\infty_0(\Omega)$, we have that
\begin{eqnarray*}
&&\int_\Omega \tilde u\,\Delta\phi =\lim_{\e\to0}
\int_{\Omega_0^+\cap\{u>\e\}} u\,\Delta\phi =
\lim_{\e\to0}
\int_{\Omega_0^+\cap\{u>\e\}} \Big(\phi\,\Delta u + 
{\rm div}\,(u\nabla\phi-\phi\nabla u)
\Big)\\
&&\qquad= 0+\lim_{\e\to0} \int_{\overline{\Om_0^+}\cap \partial\{u>\e\} }
u\frac{\partial\phi}{\partial\nu}-\phi\frac{\partial u}{\partial\nu}
=-\lim_{\e\to0}\int_{\overline{\Om_0^+}\cap \partial\{u>\e\} }
\phi\frac{\partial u}{\partial\nu}\ge0,
\end{eqnarray*}
and so~$\tilde u$ is subharmonic.
Hence the weak maximum principle
can be applied to the function~$\tilde u$, and so we conclude that, for any~$\sigma\in(0,1)$
and any~$x\in B_{\sigma r}$,
\begin{eqnarray*}
&& u(x)\chi_{\Om_0^+}(x)=\tilde u(x)\le
\fint_{B_{(1-\sigma)r}(x)} \tilde u\le
\left( \fint_{ B_{(1-\sigma)r} (x )} \tilde u^2\right)^{\frac12}
\le
\left( \frac{1}{{\mathcal{L}}^n(B_{(1-\sigma)r})} \int_{B_r} \tilde u^2\right)^{\frac12}\\&&\qquad\quad
\le
\left( \frac{{\mathcal{L}}^n(B_r\cap\Om_0^+)}{{\mathcal{L}}^n(B_{(1-\sigma)r})} \fint_{B_r\cap\Om_0^+} u^2\right)^{\frac12}
\le
\left( \frac{r^n}{(1-\sigma)^nr^n} \fint_{B_r\cap\Om_0^+} u^2\right)^{\frac12}
.\end{eqnarray*}
Thus,
\begin{equation}\label{weak-max-zzz}
\sup_{B_{\sigma r}\cap\Om_0^+} u\leq \frac{1}{(1-\sigma)^{\frac n2}}
\left(\fint_{B_r\cap\Om_0^+} u^2
\right)^{\frac12} = \frac{r\gamma}{(1-\sigma)^{\frac{n}{2}}},
\end{equation}
where we have set 
$$\gamma :=\left(\frac1{r^{2}}\fint_{B_r\cap \Om^+_0}u^2\right)^{\frac12}.$$
{F}rom this and~\eqref{KA90AK} we conclude that
\begin{equation}\label{D2}
\sup_{\partial B_{\kappa r}}|\nabla b|
\le C \gamma.\end{equation}
{F}rom \eqref{D1} and~\eqref{D2}, it follows that
\begin{equation*}
\begin{split}
& \int_{\partial B_{\kappa r}\cap\{u>\epsilon\}\cap\Omega_0^+} 
\left|\frac{\partial v_\epsilon}{\partial\nu}\right|\,(u-\epsilon) \le
\int_{\partial B_{\kappa r}\cap\{u>\epsilon\}\cap\Omega_0^+} 
\left|\nabla v_\epsilon\right|\,(u-\epsilon) \\
&\qquad\le
\int_{\partial B_{\kappa r}\cap\{u>\epsilon\}\cap\Omega_0^+} 
\left|\nabla b\right|\,(u-\epsilon) \le
C\gamma
\int_{\partial B_{\kappa r}\cap\{u>\epsilon\}\cap\Omega_0^+} 
(u-\epsilon) 
\\ &\qquad\le
C\gamma
\int_{\partial B_{\kappa r}\cap\Omega_0^+} u.
\end{split}
\end{equation*}
So, making use of \eqref{zzzz1}, we obtain 
\begin{equation}\label{TER}
\int_{B_{\kappa r}\cap \Om^+_0}\big(|\nabla u|^2+\lambda_2Q \I{u}\big)\le
C\gamma
\int_{\partial B_{\kappa r}\cap\Omega_0^+} u.
\end{equation}
Now we recall the elementary trace inequality for nonnegative functions~$f$,
see e.g. Theorem~1(ii) on page~258 of~\cite{Evans},
namely
\begin{equation}\label{JH:CA}
\int_{\partial B_{\rho}} f\le C
\,\left[\int_{B_{\rho}}|\nabla f|+\frac 1\rho\int_{B_{\rho}}f\right].
\end{equation}
We apply~\eqref{JH:CA} to~$f:=u\chi_{\Om_0^+}$ and~$\rho:=\kappa r$. Then,
since~$u$ vanishes along~$\partial\Om_0^+$,
$$\int_{\partial B_{\kappa r}\cap \Om_0^+} u\le C
\,\left[\int_{B_{\kappa r}\cap\Om_0^+}|\nabla u|+\frac 1{r}\int_{B_{\kappa r}\cap\Om_0^+} u\right],$$
where~$C>0$ now may also depend on~$\kappa$.

Hence, in light of~\eqref{TER}, we find that
\begin{equation}\label{JAK-0A}
\int_{B_{\kappa r}\cap \Om^+_0}\big(|\nabla u|^2+\lambda_2Q\I{u}\big)\le
C\gamma\left[\int_{B_{\kappa r}\cap \Om^+_0}|\nabla u|+\frac 1r\int_{B_{\kappa r}\cap \Om^+_0}u\right].
\end{equation}
Now we point out that, by the Cauchy-Schwarz inequality,
$$ 2\int_{B_{\kappa r}\cap \Om^+_0}|\nabla u|\le
\int_{B_{\kappa r}\cap \Om^+_0}\big(|\nabla u|^2+1\big).$$
This and~\eqref{JAK-0A} give that
\begin{eqnarray*}
&& \int_{B_{\kappa r}\cap \Om^+_0}\big( |\nabla u|^2+\lambda_2Q\I{u}\big)\\
&\le & C\gamma\int_{B_{\kappa r}\cap \Om^+_0}\big(|\nabla u|^2+\I{u}\big)
+ \frac{C\gamma }{r}\sup_{B_{\kappa r}\cap \Om^+_0} u\int_{B_{\kappa r}\cap \Om^+_0}\I{u}\\
&\le & C\gamma\max\left\{1, \,\frac1{Q_1\lambda_2}\right\}\left[\int_{B_{\kappa r}
\cap \Om^+_0}\big(|\nabla u|^2+\lambda_2Q\I{u}\big)
+ \frac{1 }{r}\sup_{B_{\kappa r}\cap \Om^+_0} u
\int_{B_{\kappa r}\cap \Om^+_0}Q\I{u}\right].\end{eqnarray*}
In consequence of this and~\eqref{weak-max-zzz}, we obtain that
$$ \int_{B_{\kappa r}\cap \Om^+_0}\big(|\nabla u|^2+\lambda_2Q\I{u}\big)\le
C\gamma(1+\gamma)\int_{B_{\kappa r}\cap \Om^+_0}\big(|\nabla u|^2+\lambda_2
Q\I{u}\big).$$
If $\gamma$ is sufficiently small we conclude that $u$ vanishes
identically in $B_{\kappa r}\cap\Om_0^+$, as desired.
\end{proof}

\section{Density theorems and clean ball conditions}\label{sec:density}

In this section we prove that the positive phase~$\{u>0\}$ occupies a positive
density near the free boundary points. 

\begin{thm}\label{HJsdfg}
Assume that~$u$ is a minimizer of $J$
as in~\eqref{EN:FUNCT}. 
Let~$\varpi>0$ and assume that~$
{\mathcal{L}}^n
\big(\Omega^+(u)\big)\ge\varpi$.
Let~$D\Subset \Omega$.
Assume that~$x_0\in \Gamma=\partial\Omega^+(u)$
and let~$r>0$ be such that~$B_r(x_0)\subseteq D$.

Then, there exist~$c_1\in(0,1)$, possibly depending
on~$\varpi$, $Q$, $\Om$ and~$D$,
and~$y_0\in B_r(x_0)$ such that
\begin{equation}\label{JK:DEaA1}
B_{c_1 r}(y_0)\subseteq B_r(x_0)\cap \Omega^+(u).\end{equation}
Moreover, there exists~$c_2>0$, possibly depending
on~$\varpi$, $Q$, $\Om$ and~$D$, such that
\begin{equation}\label{JK:DEaA2}
{\mathcal{L}}^n \big(B_r(x_0)\cap \Omega^+(u)\big)\ge c_2 r^n.\end{equation}
\end{thm}

\begin{proof} Obviously, we have that~\eqref{JK:DEaA2}
is a direct consequence of~\eqref{JK:DEaA1}, so we focus
on the proof of~\eqref{JK:DEaA1}. To this aim, we
recall that~$u$ is continuous, thanks to Corollary~\ref{COR:5},
hence we can take~$y_0\in\overline{ B_{r/2}(x_0)}$ such that
\begin{equation}\label{F:uIA} u(y_0)=\max_{ \overline{ B_{r/2}(x_0)} } u.\end{equation}
We take~$d:={\rm dist}(y_0,\Gamma)$ and~$z_0\in\Gamma\cap \partial{B_d(y_0)}$.
Notice that, since~$x_0\in\Gamma$, we have that
\begin{equation}\label{bAG}
d\le |x_0-y_0|\le \frac{r}2.\end{equation}
Hence, we are in the position of applying
Corollary~\ref{COR:6}, and we obtain that
\begin{equation}\label{A:p10}
u(y_0)\le Cd,
\end{equation}
for some~$C>0$.
On the other hand, from Theorem~\ref{prop-nondeg} and~\eqref{F:uIA},
$$ cr^2 \le\fint_{B_{r/2}(x_0)} u^2\le u^2(y_0),$$
for some~$c>0$. Comparing this with~\eqref{A:p10}, we conclude that
$$ d\ge \frac{u(y_0)}{C}\ge \frac{\sqrt{c} r}{C}=c_1\, r,$$
for some~$c_1>0$. As a matter of fact, by~\eqref{bAG}, 
we have that~$c_1\in(0,\,1/2)$. This construction establishes~\eqref{JK:DEaA1}.
\end{proof}

\section{Blow-up limits}\label{sec:blowup}

In this section, we consider the blow-up of a minimizer at a free
boundary point. We will show that, in the limit, we obtain a minimizer
for the Alt-Caffarelli problem in~\eqref{EN:FUNCT:AC}.
This phenomenon plays an important role in our analysis, since
it transforms the original nonlinear free boundary problem into
a linear one, in the blow-up limit: that is, in our framework,
the blow-up possesses an additional linearization feature.

To this extent,
for any~$x_0\in\Gamma$ we consider the blow-up sequence 
of~$u$ at~$x_0$, that is
\begin{equation}\label{defblow}
u_k(x) := \frac{u(x_0 + \rho_k x)}{\rho_k},\end{equation}
where~$\rho_k\to0$ as~$k\to+\infty$. 

We have the following convergence result (see e.g. Proposition~8.1 in~\cite{DK}
for the proof): 

\begin{prop}\label{tech-2}
Let $x_0\in \Gamma$ and $u_k$ a the blow-up sequence, 
as introduced in~\eqref{defblow}. 

Then there exists a blow-up limit $u_0:\R^n\to \R$, which is
continuous and with linear growth, such that, 
up to a subsequence, as~$k\to+\infty$,
\begin{itemize}
\item $u_k\to u_0$ in $C_{loc}^\alpha(\R^n)$ for any $\alpha\in (0, 1)$, 
\item $\na u_k\to \na u_0$ weakly in $L^{q}_{loc}(\R^n)$ for any $q>1$,
\item  $\fb{u_k}\to \fb{u_0}$ locally in Hausdorff distance, 
\item $\chi_{\{u_k>0\}}\to \chi_{\{u_0>0\}}$ in $L_{loc}^1(\R^n)$.
\end{itemize}
\end{prop}

We remark that in the proof of Proposition~\ref{tech-2}
we do not need Lebesgue density estimates.
The statement above can be also enhanced, giving the pointwise
convergence of the gradients, as given by the next result:

\begin{lem}\label{L15}
Let $x_0\in \Gamma$. Let~$u_k$ be the blow-up sequence, 
as introduced in~\eqref{defblow}, and~$u_0$ the blow-up limit 
given by Proposition~\ref{tech-2}. 
Then~$\nabla u_k\to \nabla u_0$
a.e. in~$\R^n$, as~$k\to+\infty$.

In addition, if~$p\in\{u_0\ne0\}$, we have that~$\nabla 
u_k\to \nabla u_0$ as~$k\to+\infty$ uniformly in a neighborhood of~$p$.
\end{lem}

\begin{proof} The proof is an appropriate modification
(and actually a simplification)
of some arguments also exploited in~\cite{AC}.
We let~${\mathscr{A}}$ be the set of the Lebesgue density points of~$\{ u_0=0\}$.
We show that~$\nabla u_k\to \nabla u_0$
in~${\mathscr{A}}\cup \{ u_0\ne0\}$, as~$k\to+\infty$,
with locally uniform convergence in~$\{ u_0\ne0\}$
(with this, since the complement of~${\mathscr{A}}\cup \{ u_0\ne0\}$
has zero Lebesgue measure, the desired result is established).

To this aim, we observe that if~$p\in \{ u_0\ne0\}$
we know from Proposition~\ref{tech-2} that there exists~$r_0>0$ such
that~$u_k(x)\ne0$ for any~$x\in B_{r_0}(p)$, as long as~$k$ is large enough.
Then, by Corollary~\ref{COR:5}, we have that~$u_k$ is harmonic in~$B_{r_0}(p)$
and so it has second derivatives estimates in~$B_{r_0/2}(p)$.
This implies that~$\nabla u_k\to \nabla u_0$ uniformly in~$B_{r_0/2}(p)$
and so, in particular, that~$\nabla u_k(p)\to \nabla u_0(p)$, as~$k\to+\infty$.

Now, let us take~$q\in{\mathscr{A}}$. Then,
$$ \lim_{r\searrow0}\frac{
{\mathcal{L}}^n(B_r(q)\cap \{u_0=0\})
}{{\mathcal{L}}^n(B_r(q))} =1$$
and therefore for any~$\eta>0$
there exists~$\bar r(\eta)>0$ such that if~$r\in(0,\bar r(\eta)]$ then
$$ \frac{{\mathcal{L}}^n(B_r(q)\cap \{u_0=0\})
}{{\mathcal{L}}^n(B_r(q))} \ge 1-\eta.$$
In particular~$ {\mathcal{L}}^n(B_r(q)\cap \{u_0\ne0\})\le \eta\,
{{\mathcal{L}}^n(B_r(q))}$ and so, in light of the Lipschitz regularity
obtained in Theorem~\ref{COR:7BIS}, we have that
$$ \fint_{ B_r(q) } u_0^2 =\frac{1}{{\mathcal{L}}^n(B_r(q))}
\int_{ B_r(q) \cap \{u_0\ne0\} } u_0^2 \le 
Cr^2 \, \frac{{\mathcal{L}}^n(B_r(q)\cap \{u_0\ne 0\})
}{{\mathcal{L}}^n(B_r(q))} \le C\eta \,r^2< \frac{c}{2}\,r^2,$$
up to renaming~$C>0$ and taking~$\eta$ suitably small,
where~$c>0$ is the one given in Theorem~\ref{prop-nondeg}.
Consequently
$$ \fint_{ B_r(q) } u_k^2 <c\,r^2$$
if $k$ is large enough, and so, by Theorem~\ref{prop-nondeg},
we have that~$u_k\le0$ in~$B_r(q)$,
and so~$u_0\le0$ in~$B_r(q)$.

We also know that~$u_k$ is subharmonic, thanks to Lemma~\ref{LEMMA2},
and thus also~$u_0$ is subharmonic. Accordingly, for small~$r>0$,
$$ 0=u_0(q)\le \fint_{B_r(q)} u_0 \le 0,$$
which implies that~$u_0$ vanishes identically in~$B_r(q)$.
Similarly, $u_k$ vanishes identically in~$B_r(q)$. These considerations
imply that~$\nabla u_k(q)=0=\nabla u_0(q)$.
\end{proof}

Next result shows that
the blow-up limit~$u_0$ is always a minimizer of the Alt-Caffarelli functional in~\eqref{EN:FUNCT:AC}
(for a suitable choice of~$Q$, which turns out to be constant).
That is, the blow-up limit has the additional, and somehow unexpected advantage, to
linearize the interfacial energy. The precise result,
which was stated in Theorem~\ref{lem-ACF-blowup}, is proved
by the following argument:

\begin{proof}[Proof of Theorem~\ref{lem-ACF-blowup}] 
The result in Theorem~\ref{lem-ACF-blowup}
can be seen as a nonlinear counterpart
of Lemma~5.4 in~\cite{AC}.
Up to a translation, we take~$x_0:=0$ in~\eqref{defblow}.
We take a competitor~$v_0$ for~$u_0$, i.e. we suppose that~$v_0-u_0\in W^{1,2}_0(B_r)$.
We also take~$\eta\in C^\infty_0(B_r,\,[0,1])$ and we define
$$ v_\rho:= v_0+(1-\eta)(u_\rho-u_0).$$
We observe that
\begin{equation}\label{92kahf:0}
v_\rho-u_\rho = (v_0-u_0) -\eta(u_\rho-u_0)\end{equation}
and so
\begin{equation}\label{92kahf}
v_\rho-u_\rho =0 {\mbox{ outside }}B_r.
\end{equation}
In addition,
$$\{ v_\rho >0 \} \subseteq \{v_0>0\}\cup \{\eta<1\}$$
and therefore
\begin{equation}\label{osdjfsdf3kn}
\chi_{\{ v_\rho >0 \}} \le \chi_{ \{v_0>0\} }+\chi_{ \{\eta<1\} }.
\end{equation}
We also define~$v(x):=\rho v_\rho(x/\rho)$.
We remark that
$$ (v-u)(x)=\rho\left( v_\rho\left( \frac{x}\rho\right)-u_\rho\left( \frac{x}\rho\right)\right)=0\quad{\mbox{
for any }}x\in\R^n\setminus B_{\rho r},$$
thanks to~\eqref{defblow} and~\eqref{92kahf}.
Since~$B_{\rho r}\subset\Omega$ when~$\rho$ is sufficiently small (possibly in dependence
of the fixed~$r>0$), we obtain that~$v-u=0$ outside~$\Omega$.

Consequently, we can use the minimality of~$u$ in~$\Omega$ and find that
\begin{equation}\label{1st0}
\begin{split}
& 0 \le J[v]-J[u]=\int_{ \Om } \big( |\nabla v|^2-|\nabla u|^2\big)+
\Phi_0\left( \lambda_2 \int_\Omega Q\chi_{\{v>0\}} \right)-
\Phi_0\left( \lambda_2 \int_\Omega Q\chi_{\{u>0\}} \right) \\
&\qquad\qquad= \int_{ B_{\rho r} } \big( |\nabla v|^2-|\nabla u|^2\big)+
\Phi_0\left( \lambda_2 \int_{ B_{\rho r} } Q\chi_{\{v>0\}} +\Xi_\rho\right)-
\Phi_0\left( \lambda_2 \int_{ B_{\rho r} } Q\chi_{\{u>0\}} +\Xi_\rho\right),
\end{split}\end{equation}
where
$$ \Xi_\rho := \lambda_2 \int_{ \Omega\setminus B_{\rho r} } Q\chi_{\{u>0\}}.$$
We point out that
\begin{equation}\label{idosj6w8qeydu2j}
\lim_{\rho\searrow0}\Xi_\rho=\Xi_0:=
\lambda_2 \int_{ \Omega } Q\chi_{\{u>0\}}\,.
\end{equation}
Now we scale the quantities in~\eqref{1st0}, using the substitution~$y:=x/\rho$.
In this way, we find that
\begin{eqnarray*} &&\int_{ B_{\rho r} } |\nabla v(x)|^2\,dx=
\int_{ B_{\rho r} } |\nabla v_\rho(x/\rho)|^2\,dx =
\rho^n\int_{ B_r } |\nabla v_\rho(y)|^2\,dy\\
{\mbox{and }}&&
\int_{ B_{\rho r} } Q(x)\chi_{\{v>0\}}(x)\,dx=
\int_{ B_{\rho r} } Q(x)\chi_{\{v_\rho>0\}}(x/\rho)\,dx=
\rho^n \int_{ B_r } Q(\rho y)\chi_{\{v_\rho>0\}}(y)\,dy,
\end{eqnarray*}
and similar expressions hold true with~$u$ replacing~$v$. Substituting these identities into~\eqref{1st0},
we conclude that
\begin{equation}\label{1st1}
\begin{split}&
0\le \int_{ B_r } \big( |\nabla v_\rho|^2-|\nabla u_\rho|^2\big)\\ &\qquad\qquad+\frac{1}{\rho^n}\left[
\Phi_0\left( \rho^n\lambda_2 \int_{ B_r } Q(\rho x)\chi_{\{v_\rho>0\}}(x)\,dx +\Xi_\rho\right)-
\Phi_0\left( \rho^n\lambda_2 \int_{ B_r } Q(\rho x)\chi_{\{u_\rho>0\}}(x)\,dx +\Xi_\rho\right)\right].\end{split}
\end{equation}
Now, from~\eqref{92kahf:0}, we have that
$$ v_\rho+u_\rho = v_0-u_0 -\eta(u_\rho-u_0)+2u_\rho
= (v_0+u_0)+(2-\eta)(u_\rho-u_0).$$
This and~\eqref{92kahf:0} give that
\begin{eqnarray*} && |\nabla v_\rho|^2-|\nabla u_\rho|^2
= \nabla(v_\rho+u_\rho)\cdot(v_\rho-u_\rho)\\&&\qquad=
\nabla \big( 
(v_0+u_0)+(2-\eta)(u_\rho-u_0)
\big)\cdot\nabla\big( (v_0-u_0) -\eta(u_\rho-u_0)\big)\\
&&\qquad=\nabla(v_0+u_0)\cdot\nabla(v_0-u_0) 
-\eta\nabla(v_0+u_0)\cdot\nabla(u_\rho-u_0)
-(u_\rho-u_0)\nabla(v_0+u_0)\cdot\nabla\eta
\\ &&\qquad\qquad\qquad-(u_\rho-u_0)\nabla\eta\cdot\nabla(v_0-u_0)
+(2-\eta)\nabla(u_\rho-u_0)
\cdot\nabla(v_0-u_0)
\\ &&\qquad\qquad\qquad
+|u_\rho-u_0|^2\,|\nabla\eta|^2
+\eta(u_\rho-u_0)\nabla\eta\cdot\nabla(u_\rho-u_0)
\\ &&\qquad\qquad\qquad
-(2-\eta)(u_\rho-u_0)\nabla(u_\rho-u_0)\cdot\nabla\eta
-(2-\eta)\eta\,|\nabla(u_\rho-u_0)|^2.
\end{eqnarray*}
We remark that the latter term has a sign. So, recalling Proposition~\ref{tech-2},
we obtain that
\begin{equation}\label{0qwu12}
\lim_{\rho\searrow0}\int_{ B_r } \big( |\nabla v_\rho|^2-|\nabla u_\rho|^2\big)
\le 
\int_{ B_r } \nabla(v_0+u_0)\cdot\nabla(v_0-u_0)=
\int_{ B_r } \big( |\nabla v_0|^2-|\nabla u_0|^2\big).
\end{equation}
Now we set
\begin{eqnarray*}
&& \alpha_\rho:=
\lambda_2 \int_{ B_r } Q(\rho x)\chi_{\{v_0>0\}}(x)\,dx+
\lambda_2 \int_{ B_r } Q(\rho x)\chi_{\{\eta<1\}}(x)\,dx\\
{\mbox{and }}&&
\beta_\rho:=
\lambda_2 \int_{ B_r } Q(\rho x)\chi_{\{u_\rho>0\}}(x)\,dx
.\end{eqnarray*}
We observe that
\begin{eqnarray*}
&& \lim_{\rho\searrow0}\alpha_\rho=\alpha_0:=
\lambda_2 \int_{ B_r } Q(0)\chi_{\{v_0>0\}}(x)\,dx+
\lambda_2 \int_{ B_r } Q(0)\chi_{\{\eta<1\}}(x)\,dx\\
{\mbox{and }}&&\lim_{\rho\searrow0}
\beta_\rho=\beta_0:=
\lambda_2 \int_{ B_r } Q(0)\chi_{\{u_0>0\}}(x)\,dx,
\end{eqnarray*}
thanks to Proposition~\ref{tech-2}.

Then, recalling~\eqref{osdjfsdf3kn} and the monotonicity of~$\Phi_0$, 
and exploiting also~\eqref{idosj6w8qeydu2j}, we have that
\begin{eqnarray*}
&&\lim_{\rho\searrow0}
\frac{1}{\rho^n}\left[
\Phi_0\left( \rho^n\lambda_2 \int_{ B_r } Q(\rho x)\chi_{\{v_\rho>0\}}(x)\,dx +\Xi_\rho\right)-
\Phi_0\left( \rho^n\lambda_2 \int_{ B_r } Q(\rho x)\chi_{\{u_\rho>0\}}(x)\,dx +\Xi_\rho\right)\right]\\
&\le& \lim_{\rho\searrow0}
\frac{1}{\rho^n}\Big[\Phi_0(\rho^n\alpha_\rho+\Xi_\rho)-
\Phi_0(\rho^n\beta_\rho+\Xi_\rho) \Big]\\
&=&\lim_{\rho\searrow0}
(\alpha_\rho-\beta_\rho )\int_0^1 \Phi_0' \big(\rho^n\beta_\rho+t
\rho^n(\alpha_\rho-\beta_\rho)
+\Xi_\rho\big)\,dt
\\ &=&
(\alpha_0-\beta_0 )\, \Phi_0' (\Xi_0).
\end{eqnarray*}
So, we insert this inequality and~\eqref{0qwu12} into~\eqref{1st1}
and we obtain
\begin{eqnarray*}
0 &\le& \int_{ B_r } \big( |\nabla v_0|^2-|\nabla u_0|^2\big)
+(\alpha_0-\beta_0 )\, \Phi_0' (\Xi_0) \\
&=& \int_{ B_r } \big( |\nabla v_0|^2-|\nabla u_0|^2\big)
+
\lambda_2 Q(0)\,\Phi_0' (\Xi_0)\,\int_{ B_r } \chi_{\{v_0>0\}}(x)\,dx\\ &&\quad+
\lambda_2 Q(0)\,\Phi_0' (\Xi_0)\,\int_{ B_r } \chi_{\{\eta<1\}}(x)\,dx
-\lambda_2 Q(0)\,\Phi_0' (\Xi_0)\,\int_{ B_r } \chi_{\{u_0>0\}}(x)\,dx
.
\end{eqnarray*}
This estimate is valid for any choice of the function~$\eta$; therefore,
letting~$\{\eta=1\}$ invade the whole of~$B_r$, we deduce that
the term~$\int_{ B_r } \chi_{\{\eta<1\}}(x)\,dx$ can be made as small
as we wish. As a consequence,
\begin{eqnarray*}
0 &\le& \int_{ B_r } \big( |\nabla v_0|^2-|\nabla u_0|^2\big)
+
\lambda_2 Q(0)\,\Phi_0' (\Xi_0)\,\int_{ B_r } \chi_{\{v_0>0\}}(x)\,dx
-\lambda_2 Q(0)\,\Phi_0' (\Xi_0)\,\int_{ B_r } \chi_{\{u_0>0\}}(x)\,dx
,
\end{eqnarray*}
which establishes the desired minimality property for~$u_0$.
\end{proof}

\section{Partial regularity of the free boundary}\label{sec:partialreg}

Using the subharmonicity property of the minimizers and their
Lipschitz regularity (recall Lemma~\ref{LEMMA2}
and Theorem~\ref{COR:7BIS}), we are now in the position to exploit
standard techniques from geometric measure theory and
conclude a partial regularity of the free boundary, as claimed in Theorem~\ref{GMT-00}:

\begin{proof}[Proof of Theorem~\ref{GMT-00}] The proof of~(i)
in Theorem~\ref{GMT-00} follows by a standard integration by parts
(combined with the subharmonicity property of Lemma~\ref{LEMMA2},
see e.g. the proof of formula~(8.1) in~\cite{DK} for details).

To prove~(ii) in Theorem~\ref{GMT-00},
we argue by contradiction and we suppose that
there exist~$r_j\searrow0$ and~$x_j\in\Gamma$ such that~$B_{2r_j}(x_j)\subset\Omega$
and
\begin{equation}\label{1AOJ} \int_{B_{r_j}(x_j)}\Delta u^+\le\frac{r_j^{n-1}}{j}.\end{equation}
Up to subsequences, we may suppose that~$x_j\to \bar x\in\overline{D}\subset\Omega$.
We define
$$v_j(x):=\frac{u(x_j+r_j x)}{r_j}.$$
Notice that, in the light of Corollary~\ref{COR:7}, for any~$x$, $y\in\overline{ B_1}$ we have
$$ |v_j(x)-v_j(y)|\le
\frac{\big|u(x_j+r_j x)-u(x_j+r_jy)\big|}{r_j}\le\frac{C\,|(x_j+r_j x)-(x_j+r_jy)|}{r_j}\le
C\,|x-y|.$$
So $v_j$ (and hence~$v_j^+$)
is Lipschitz in~$\overline{B_1}$ uniformly in~$j$ and thus we may suppose
that~$v_j^+$ converges to some~$\bar v^+$ uniformly in~$\overline{B_1}$.

Now, let~$\phi\in C^\infty_0\big(B_1,\,[0,+\infty)\big)$. Notice that, by
Lemma~\ref{LEMMA2}, we know that~$u$ is subharmonic,
hence so is~$u^+$ and~$\Delta u^+\ge0$. Hence, recalling~\eqref{1AOJ}, we have
\begin{eqnarray*}&&
\int_{B_1} \bar v^+\Delta\phi = \lim_{j\to+\infty}\int_{B_1} v^+_j\Delta\phi
= \lim_{j\to+\infty} r_j^{-1}\int_{B_1} u^+(x_j+r_j x)\,\Delta\phi(x)\,dx\\
&&\qquad=
\lim_{j\to+\infty} r_j \int_{B_1} \Delta u^+(x_j+r_j x)\,\phi(x)\,dx
=
\lim_{j\to+\infty} r_j^{1-n} 
\int_{B_{r_j}(x_j)} \Delta u^+(\xi)\,\phi\left( \frac{\xi-x_j}{r_j}\right)\,d\xi\\
&&\qquad \le \sup_{\R^n} \phi\,
\lim_{j\to+\infty} r_j^{1-n} 
\int_{B_{r_j}(x_j)} \Delta u^+(\xi)\,d\xi
\le\sup_{\R^n} \phi\,
\lim_{j\to+\infty}\frac{1}{j}=0.
\end{eqnarray*}
Accordingly, $\bar v^+$ is superharmonic in~$B_1$. By construction, we also have that~$\bar v^+\ge0$ and
$$ \bar v^+(0)=\lim_{j\to+\infty}\frac{u(x_j)}{r_j}
=\lim_{j\to+\infty}\frac{0}{r_j}=0.$$
As a consequence
\begin{equation}\label{aI12adA}
{\mbox{$\bar v^+$ vanishes identically in~$B_1$.}}\end{equation}
On the other hand, by \eqref{PNOAL9},
$$ cr_j^2 \le \fint_{B_{r_j}(x_j)\cap \Om^+_0}u^2\le \sup_{B_{r_j}(x_j)} u^2\le r_j^2
\sup_{B_1} v_j^2,$$
for some~$c>0$. So, simplifying~$r_j$ on both sides of the inequality and taking
the limit in~$j$, we find that
$$ c\le \sup_{B_1} \bar v^2.$$
This is in contradiction with~\eqref{aI12adA}
and so the proof of~(ii) is complete.

Then, \eqref{90AKJA:A} follows from point~(ii) and suitable geometric measure
theory arguments (see e.g. the proof of Corollary~8.2 in~\cite{DK}
for full details). For the proof of~\eqref{90AKJA:B}, see e.g.
the proof of Theorem~B in~\cite{DK}.
\end{proof}

\section{Regularity of the free boundary}\label{sec:FBregularity2D}

In this section we show that at flat points the free boundary is a 
regular smooth surface. In particular, 
in two spatial dimensions the free boundary is a 
continuously differentiable curve. 

Our approach differs from the one in \cite{AC, ACF}
as we avoid using the flatness classes. Instead, we use the free boundary 
regularity theory for the viscosity solutions from~\cite{DK}.

\subsection{Viscosity solutions}
Recall the definitions of~$\Omega^+(u)$ and ~$\Omega^-(u)$
given in the section with the notation. 

If the free boundary is $C^1$ smooth, then
$$ G(u^+_\nu ,u^-_\nu ) := (u^+_\nu )^2 - (u^-_\nu )^2 -\Lambda$$
is the flux balance across the free boundary, 
where~$u^+_\nu$ and~$u^-_\nu$ are the normal derivatives in the
inward direction to~$\partial\Omega^+(u)$ and~$\partial\Omega^-(u)$, 
respectively, and $\Lambda$ is defined in \eqref{eq-big-lamb}. 

With this notation, we give the definition of viscosity solution:

\begin{defn}\label{def:visc}
Let~$\Omega$ be a bounded domain of~$\R^n$ and let~$u$ 
be a continuous function in~$\Omega$. We say that~$u$ is a viscosity solution
in~$\Omega$ if
\begin{itemize}
\item[i)] $\Delta u=0$ in~$\Omega^+(u)$ and~$\Omega^-(u)$,
\item[ii)] along the free boundary~$\Gamma$, $u$ satisfies the free boundary condition, in the sense that:
\begin{itemize}
\item[a)] if at~$x_0\in\Gamma$ there exists a ball~$B\subset\Omega^+(u)$
such that~$x_0\in \partial B$ and
\begin{equation*}
u^+(x)\ge\alpha\langle x-x_0,\nu\rangle^+ + o(|x-x_0|), \ {\mbox{ for }} x\in B,
\end{equation*}
\begin{equation*}
u^-(x)\le\beta\langle x-x_0,\nu\rangle^- + o(|x-x_0|), \ {\mbox{ for }} x\in B^c,
\end{equation*}
for some $\alpha>0$ and~$\beta\ge0$, with equality along every non-tangential domain,
then the free boundary condition is satisfied
$$ G(\alpha,\beta)=0, $$
\item[b)] if at~$x_0\in\Gamma$ there exists a ball~$B\subset\Omega^-(u)$
such that~$x_0\in \partial B$ and
$$ u^-(x)\ge\beta\langle x-x_0,\nu\rangle^- + o(|x-x_0|), \ {\mbox{ for }} x\in B, $$
$$ u^+(x)\le\alpha\langle x-x_0,\nu\rangle^+ + o(|x-x_0|), \ {\mbox{ for }} x\in\partial B, $$
for some $\alpha\ge0$ and~$\beta>0$, with equality along every non-tangential domain,
then
$$ G(\alpha,\beta)=0. $$
\end{itemize}
\end{itemize}
\end{defn}

\begin{lem}\label{lem-visc-sol-full}
Let $u$ be a minimizer in~$\Omega$ for the functional~$J$ in~\eqref{EN:FUNCT}.
Then, $u$ is also a viscosity solution in the sense of Definition \ref{def:visc}.
\end{lem}

\begin{proof}
See Lemma~11.17 in~\cite{CS} or Theorem~4.2 in~\cite{DK} for the proof.
\end{proof}

Notice that, if $x_0\in \fbr{u}$, then $\Gamma$ is flat near $x_0$. 
Therefore, since a minimizer~$u$ is also a viscosity solution, 
according to Lemma~\ref{lem-visc-sol-full}, we can use 
the Harnack inequality approach to $u$, and obtain that~$\Gamma$ 
is $C^{1, \alpha}$ in some neighborhood of~$x_0$. 

We next show that in two dimensions the free boundary is 
a continuously differentiable curve.

\subsection{The case in which $u^-$ is degenerate}
In this section we show that near the points $x_0\in \Gamma$
where $u^-$ is degenerate, the minimizer $u$ behaves 
essentially as a solution to the 
one-phase problem. Recall that we say that~$u^-$ is degenerate at~$x_0$ if
\begin{equation}\label{DEG}
\liminf_{r\to0}  \frac{1}{r}\fint_{B_r(x_0)}u^- =0. \end{equation}
We stress that~$u^+$ (in contrast to ~$u^-$) is always nondegenerate,
according to the following observation:

\begin{lem}\label{upng}
Let~$u$ be a minimizer in~$\Omega$ 
for the functional~$J$ in~\eqref{EN:FUNCT}, 
with~$0\in \fb{u}$. 
Assume that~\eqref{IOTA:0} and~\eqref{IOTA} are satisfied.
Then
$$ \liminf_{r\to0}  \frac{1}{r}\fint_{B_r}u^+>0.$$
\end{lem}

\begin{proof} By the clean ball condition in Theorem~\ref{HJsdfg},
we have that~$B_{c_1 r}(y_0)\subseteq B_r\cap \Om^+(u)$,
for a suitable point~$y_0\in B_r$ and a constant~$c_1>0$.

Thus, from Lemma~\ref{BY BELOW}, we have that~$u(y)\ge c_2 r$
for any~$y\in B_{c_1 r/2}(y_0)$, for some~$c_2>0$. As a consequence,
$$ \int_{B_r} u^+ \ge \int_{B_{c_1 r/2}(y_0)} u^+ \ge c_2 r\;{\mathcal{L}}^n(
B_{c_1 r/2}(y_0)),$$
which gives the desired result.
\end{proof}

Next, we show that if $u^-$ is degenerate then so is the gradient of $u^-$ in the following sense: 
\begin{lem}\label{0-grad} 
Let~$u$ be a minimizer in~$\Omega$ for the functional~$J$ in~\eqref{EN:FUNCT}.
Let~$x_0\in \fb{u}$, and suppose that~$u^-$ is degenerate at~$x_0$.
Then 
$$ \lim_{r\to 0}\sup_{x\in B_r(x_0)}|\na u^-(x)|=0.$$
\end{lem}

\begin{proof}
We argue by contradiction and we suppose that 
the conclusion of the lemma fails. 
Then, there exists a sequence $r_j\to 0$, as~$j\to+\infty$, 
such that 
\begin{eqnarray}\label{ghgeasd96785}
&& \lim_{j\to+\infty}\frac1{r_j}\fint_{B_{r_j}(x_0)}u^-= 0\\
{\mbox{and }} && \label{810}
\lim_{j\to +\infty}\sup_{B_{r_j}(x_0)}|\na u^-|>0.\end{eqnarray}
Consider the scaled functions $u_j(x):=\frac{u(x_0+r_j x)}{r_j}$. 
{F}rom Theorem~\ref{COR:7BIS} we obtain that~$|\nabla u_j|$ is bounded
uniformly in~$j$ in~$B_1$. Then, up to a subsequence,
we have that $u_j\to u_0$ as~$j\to+\infty$
uniformly in~$B_1$, for some function~$u_0$.  Moreover,
by~\eqref{ghgeasd96785},
$$ 0=
\lim_{j\to+\infty}\frac1{r_j}\fint_{B_{r_j}(x_0)}u^-
=\lim_{j\to+\infty}\fint_{B_1}u^-_j =\fint_{B_1}u_0^-.$$
This implies that
\begin{equation}\label{16}
{\mbox{$u_0^-=0$ in $B_1$.}}
\end{equation}
In addition, by~\eqref{810} and Lemma~\ref{L15},
$$ 0<
\lim_{j\to +\infty}\sup_{B_{r_j}(x_0)}|\na u^-|
=
\lim_{j\to +\infty}\sup_{B_1}|\na u^-_j|=\sup_{B_1}|\na u_0^-|
,$$
which is in contradiction with~\eqref{16}.
\end{proof}
From Lemmata~\ref{upng} and~\ref{0-grad} it follows that, 
near a degenerate point~$x_0\in \Gamma$, 
$u$ behaves almost like a minimizer of a one-phase functional.
It is well-known that for the one-phase the gradient $|\na u|$ is 
upper semicontinuous.
The aim of the next two lemmata is to establish this property of the gradient near 
degenerate points. 

First we recall the Bernoulli constant 
\begin{equation}\label{12.5bis}
\Lambda(x_0):=\left[\lambda_2\partial_{r_2}\Phi
\left(\lambda_1 \int_\Om Q\,\chi_{ \{ u<0\} },\;
\lambda_2\int_\Om Q \,\chi_{ \{ u>0\} }\right)\,
-\, \lambda_1\partial_{r_1}\Phi \left(\lambda_1
\int_\Om Q \,\chi_{ \{ u<0\} },\; \lambda_2
\int_\Om Q \,\chi_{ \{ u>0\} }\right)\,
\right]\,Q(x_0)
\end{equation}
measuring the gradient jump across the free boundary.
We observe that, in view of Lemma~\ref{FBCOND} and using the notation in~\eqref{no00n},
if~$x_0$ is a smooth point of~$\partial\{u>0\}$, we know that
\begin{equation}\label{no00n:2}
\big( \partial_\nu^+ u(x_0)\big)^2 - 
\big( \partial_\nu^- u(x_0)\big)^2=\Lambda(x_0).
\end{equation}
Furthermore, from Lemma \ref{0-grad} we should get that 
$(\partial_\nu u^+(x))^2=\Lambda(x_0)+o(1)$ as $x\to x_0$.
The next lemma\footnote{We point out that Lemma~\ref{lem-Qx}
is not explicitly used anywhere in this paper, but we included it
for completeness, since it clarifies the picture of the case
in which~$u^-$ is degenerate.} makes this statement precise. 

\begin{lem}\label{lem-Qx}
Let~$u$ be a minimizer in~$\Omega$ for the functional~$J$ in~\eqref{EN:FUNCT}.
Let~$x_0\in \fb{u}$, and suppose that~$u^-$ is degenerate at~$x_0$. Then,
\begin{equation}\label{labgrad}
\limsup_{x\to x_0}|\na u(x)|^2=\Lambda(x_0). \end{equation}
\end{lem}

\begin{proof} 
Let us denote by 
\begin{equation}\label{LISU}
\gamma:=\limsup_{x\to x_0}|\na u(x)|.\end{equation} 
By Lemma~\ref{0-grad},
$$
\gamma=\limsup_{x\to x_0}|\na u^+(x)|,$$
hence, in order to prove~\eqref{labgrad}, one has to show that 
\begin{equation}\label{qw9e}
\gamma^2=\Lambda(x_0).\end{equation} 
For this, we take a sequence $x_k\to x_0$ such that~$x_k\in\{u>0\}$
and~$|\na u^+(x_k)|\to \gamma$
as~$k\to+\infty$.

Let $\rho_k:=\dist(x_k, \fb u)$ and let $y_k\in \fb u$ such that 
$\rho_k=|x_k-y_k|$. Consider the blow-up sequence 
$u_k(x):=\frac{u(y_k+\rho_k x)}{\rho_k}$.
{F}rom Proposition~\ref{tech-2}, up to a subsequence, we may assume
that $u_k\to u_0$ as~$k\to+\infty$ locally uniformly.

Without loss of generality we can also assume that
$$\frac{x_k-y_k}{\rho_k}\to -e_n, \quad {\mbox{as }}k\to+\infty,$$
where $e_n$ is the unit direction of the $x_n$ axis.
Thus we have that 
\begin{equation}\label{B1} B_1(-e_n)\subseteq \{u_0>0\}.\end{equation}
This and Lemma~\ref{L15} give that
\begin{equation}\label{78uyja1h} \gamma=\lim_{k\to+\infty}|\nabla u^+(x_k)|=
\lim_{k\to+\infty}\left|\nabla u_k^+\left( \frac{x_k-y_k}{\rho_k}\right)\right|
=|\na u_0(-e_n)|.\end{equation}
{F}rom~\eqref{B1}, we also obtain that
\begin{equation}\label{8945}
{\mbox{$u_0$ is harmonic in $B_1(-e_n)$,}}\end{equation}
thanks to Lemma~\ref{LEMMA2}
and Proposition~\ref{tech-2}.

We also observe that
\begin{equation}\label{LISU:3}
|\na u_0|\le \gamma\ \mbox{in}\ B_1(-e_n).\end{equation}
Indeed, if~$\bar x\in B_1(-e_n)$, we write~$\bar x=-e_n+z$, with~$|z|<1$
and we set
$$ z_k:= y_k+\rho_k \bar x= y_k+ \rho_k z-\rho_k e_n
=\rho_k \left( \frac{y_k-x_k}{\rho_k} -e_n\right)+\rho_k z+x_k\to x_0,$$
as~$k\to+\infty$. Thus, by~\eqref{LISU},
\begin{equation}\label{LISU:2}
\gamma=\limsup_{x\to x_0}|\na u(x)|\ge
\lim_{k\to+\infty} |\na u(z_k)|.\end{equation}
On the other hand,
$$ \nabla u_k(\bar x)=\nabla u(y_k+\rho_k \bar x)=\nabla u(z_k).$$
Hence, taking the limit as~$k\to+\infty$
(and recalling~\eqref{B1} and Lemma~\ref{L15}), we see that
$$ |\nabla u_0(\bar x)|=\lim_{k\to+\infty} |\nabla u(z_k)|.$$
This and~\eqref{LISU:2} imply \eqref{LISU:3}, as desired.

We also remark that
\begin{equation}
\label{g00}
\gamma>0.\end{equation}
Indeed, if~$\gamma=0$,
it follows from~\eqref{LISU:3} that~$u_0$ is constant in~$B_1(-e_n)$.
Thus, since
\begin{equation}\label{i0osd}
u_0(0)=\lim_{k\to+\infty} u_k(0)=\lim_{k\to+\infty} \frac{u(y_k)}{\rho_k}=0,\end{equation}
we obtain that~$u_0$ vanishes identically in~$B_1(-e_n)$,
in contradiction with~\eqref{B1}, thus proving~\eqref{g00}.

Now, we claim that
\begin{equation}\label{812}
u_0(x)=-\nabla u_0(-e_n)\cdot x {\mbox{ for any }}x\in B_1(-e_n).
\end{equation}
For this, we argue as follows:
by~\eqref{78uyja1h} and~\eqref{g00}, we can define
$$ \ell:= -\frac{\na u_0(-e_n)}{\gamma}=
-\frac{\na u_0(-e_n)}{|\na u_0(-e_n)|}.$$
Then, by~\eqref{78uyja1h}
and~\eqref{LISU:3}, we find that
\begin{equation}\label{8912}
\partial_\ell u_0(-e_n)=-\gamma \quad{\mbox{ and }}\quad \partial_\ell u_0(x)\ge-\gamma\ \mbox{in}\ B_1(-e_n).\end{equation}
Furthermore, in light of~\eqref{8945}, we know that~$\partial_\ell u_0$
is harmonic in~$B_1(-e_n)$. This, \eqref{8912} and the strong maximum
principle imply that~$\partial_\ell u_0=-\gamma$ in~$B_1(-e_n)$.

Then, we take a rotation~${\mathcal{R}}$ such that~$e_1={\mathcal{R}} \ell$
and we define~$ v_0(x):= u_0({\mathcal{R}} x)$.
We have that
$$\partial_1 v_0(x)= 
\big( {\mathcal{R}}\nabla u_0({\mathcal{R}} x)\big)\cdot e_1
=\nabla u_0({\mathcal{R}} x)\big)\cdot\ell=-\gamma$$
for any~$x$ such that~${\mathcal{R}} x\in B_1(-e_n)$.

Consequently, for any~$x$ such that~${\mathcal{R}} x\in B_1(-e_n)$,
we have that
\begin{equation} \label{t7}
v_0(x)=-\gamma x_1 + \tilde v(x_2,\dots,x_n),\end{equation}
for some~$\tilde v:\R^{n-1}\to\R$.
In particular,
\begin{equation}\label{09}
|\nabla v_0|^2 = |\gamma|^2 +\sum_{i=2}^{n}|\partial_i \tilde v|^2.\end{equation}
On the other hand, by~\eqref{LISU:3},
for any~$x$ such that~${\mathcal{R}} x\in B_1(-e_n)$,
$$ |\nabla v_0(x)|^2=
\big| {\mathcal{R}}\nabla u_0({\mathcal{R}} x)\big|^2\le
|\nabla u_0({\mathcal{R}} x)|^2\le\gamma.$$
Then we insert this into~\eqref{09} and we obtain that~$\partial_i \tilde v$
vanishes identically for any~$i\in\{2,\dots,n\}$, hence~$\tilde v$
is constant, and~\eqref{t7} reduces to
$$ v_0(x)=-\gamma x_1 +\tilde c,$$
for some~$\tilde c\in\R$.

Now, we recall~\eqref{i0osd} and we obtain that
$$ 0=u_0(0)=v_0(0)=\tilde c.$$
As a consequence, we obtain that
$$v_0(x)=-\gamma x\cdot e_1=-\gamma x\cdot {\mathcal{R}}^T \ell=
-\gamma ({\mathcal{R}} x)\cdot \ell.$$
Hence 
$$ u_0(x)=v_0( {\mathcal{R}}^T x)=-\gamma x\cdot \ell,$$
thus completing the proof of~\eqref{812}.

Now, from~\eqref{B1} and~\eqref{812}, we deduce that~$\ell=e_n$,
and therefore
\begin{equation}\label{PF:1}
u_0(x)=-\gamma x_n {\mbox{ in }} B_1(-e_n).
\end{equation}
We claim that, in fact,
\begin{equation}\label{PF:2}
u_0(x)=-\gamma x_n {\mbox{ in }} \{x_n<0\}.
\end{equation}
To check this, we recall~\eqref{B1} and we denote by~${\mathcal{C}}$
the connected component of~$\{u_0>0\}$ that contains~$B_1(-e_n)$.
By Corollary~\ref{COR:5}, we know that~$u_0$ is harmonic in~${\mathcal{C}}$.
Hence, by~\eqref{PF:1} and the unique continuation principle, we obtain that~$u_0(x)=-\gamma x_n$
in~${\mathcal{C}}$. As a consequence, since~$u_0$ vanishes along~$ \partial {\mathcal{C}}$,
we have that 
$$ \partial {\mathcal{C}}\subseteq \{-\gamma x_n=0\} =\{x_n=0\}, $$
thanks to~\eqref{g00}, and this establishes~\eqref{PF:2}.

It remains to show that 
$\fb{ u_0}=\{y_n=0\}$. 
To see this it is enough to show that there exists~$\delta>0$ such that
\begin{equation}\label{no00n:3}
{\mbox{$u_0=0$ in $\{x_n\in(0,\delta)\}$.}}
\end{equation}
Suppose that 
\[s=\limsup_{\begin{subarray}
yy_n\searrow 0,\\ y'\in\R^{n-1},\\ u_0(y', y_n)>0
\end{subarray}
}\frac{\partial u_0(y', y_n)}{\partial y_n}.
\]
Note that $u_0$ is a minimizer of ACF functional, see Theorem~\ref{lem-ACF-blowup}. 
Taking a sequence $\frac{\partial u_0(y'_k, h_k)}{\partial y_n}\to 0
$ as $h_k\to 0$ and using the same argument above it follows that 
the second blow of $u_0$, which we call $u_{00}$, with respect to the balls 
$B_{h_k}(y_k', 0)$ is of the form~$u_{00}=s y_n$, with~$y_n>0$.
This is a contradiction, since the zero set of 
the minimizers of ACF functional has nontrivial measure.
Thus it follows that $s=0$ and consequently 
we have that $u_0=0$ in some strip $\{0<y_n<\delta\}$, for a suitable~$\delta>0$.
This establishes \eqref{no00n:3}.

Now, in light of~\eqref{no00n:2}, \eqref{PF:2}, 
and~\eqref{no00n:3}, we have that
$$ \Lambda(x_0)=
\big( \partial_\nu^+ u_0(0)\big)^2 - \big( \partial_\nu^- u_0(0)\big)^2=\gamma^2-0,$$
which proves~\eqref{qw9e},
as desired.
\end{proof}

Now we prove a flatness result in dimension~$2$
for the case in which~$u^-$ is degenerate.

\begin{prop}\label{NEWPF}
Let $u$ be as in Lemma \ref{lem-Qx}. If $n=2$ then the free boundary of $u$ is flat at every point. 
\end{prop}

\begin{proof}
The proof is based on a compactness argument.
Consider $u_k(x)=\frac{u(x_0+r_k x)}{r_k}$ for some positive sequence $r_k\searrow 0$ with $x_0\in\partial\{u>0\}$.

Since, by Theorem~\ref{COR:7BIS},
$u$ is Lipschitz continuous, it follows that $\{u_k\}$ is locally Lipschitz.
So, in view of Proposition \ref{tech-2}, we can employ a customary compactness argument to show that there exists a subsequence $\{u_k\}$ 
converging to a limit $u_0\in W^{1, \infty}_{\rm loc}(\R^2)$
such that
  \begin{eqnarray}\label{HD0}
   u_k\rightarrow u_0 \qquad \textrm{ strongly in} \ W^{1, p}_{\rm loc}(\R^2), \forall p>1\  \textrm{ and } C^\alpha_{\rm loc}(\R^n), \forall \alpha\in (0,1) \
\textrm{as}\ k\rightarrow \infty,\\\label{HD}
   \partial \{ u_k >0 \} \rightarrow \partial \{ u_0 >0 \}\qquad \textrm{in Hausdorff distance}\  d_\H\textrm { locally in}\   \R^2, \\\label{HD1}
   \I{u_k}\rightarrow \I{u_0}\qquad  \textrm{in}\  L^1_{\rm loc}(\R^2).
  \end{eqnarray}
For the proofs of \eqref{HD0}-\eqref{HD1} we refer the reader for instance to \cite{DP-2D} and \cite{Wolan-adv}.

Let $u_0$ be a blow-up of $u$ at $x_0\in \fb u$.
Then, by Theorem~\ref{lem-ACF-blowup}
and the degeneracy of~$u^-$, it holds that $u_0\ge 0$ is a minimizer of the Alt-Caffarelli functional in $\R^2$
with the constant $\lambda_0=\lambda_0(u)$ given in~\eqref{di u}. 
Hence, we now claim that 
\begin{equation}\label{POjAI10}
u_0=\lambda_0x_1^+\end{equation} in a appropriate coordinate system
(in which~$x_0$ is the origin). Indeed, let 
\[W(x, r, u_0)=\frac1{r^{2}}\int_{B_r(x)}|\nabla u_0|^2+\lambda_0^2\I {u_0}-\frac1{r^4}\int_{\p B_r(x)}u^2_0\]
be the Weiss energy. {F}rom the monotonicity of $W$ (see~\cite{Weiss})
it follows that if we first let $r\to 0$, and then $r\to \infty$ for fixed $s>0$  then 
\[
W(0, s, u_{00})\le W(0, 0, u_0)\le W(0, sr, u_0)\le  W(0, \infty, u_0)=W(0, s, u_{0\infty})
,\]
where $u_{00}$ is a blow-up of $u_0$ at the origin and $u_{0\infty}$
a blow-down of $u_0$ at $0$, namely
$$u_{0\infty}=\lim_{R_k\to \infty}\frac{u_0(R_kx)}{R_k}.$$ Both the
blow-up and the blow-down are well-defined for 
$u_0$ because it is a minimizer of the Alt-Caffarelli functional, see \cite{Weiss}.
But $W(0, s, u_{00})=W(0, s, u_{0\infty})={\rm const}$, since in a suitable coordinate system both functions $u_{00}, u_{0\infty}$
will have the form $\lambda_0 x_1^+$.
Hence $u_0$ must be homogeneous and therefore linear, thus proving~\eqref{POjAI10}, which in turn gives the desired result. 
\end{proof}

\subsection{The case in which $u^-$ is nondegenerate}

If $u^-$ is nondegenerate then the Alt-Caffarelli-Friedman functional
\begin{equation}
\phi(r, x_0, u)=\frac1{r^4}\int_{B_r(x_0)}\frac{|\na u^+|^2}{|x-x_0|^2}dx\int_{B_r(x_0)}\frac{|\na u^-|^2}{|x-x_0|^2}dx
\end{equation}
has positive limit and therefore from Theorem 7.4 (i) in~\cite{ACF}
when 
$n=2$ the blow-up $u_0$ must be a two-plane solution.

\begin{lem}\label{lem-nax-verjin}
Let~$u$ be a minimizer in~$\Omega$ for the functional~$J$ in~\eqref{EN:FUNCT}, with~$0\in \fb{u}$. Assume that~\eqref{IOTA:0} and~\eqref{IOTA} are satisfied.

Let $\{r_j\}_{j\in\N}$ be a
sequence of positive numbers such that 
$r_j\searrow 0$ as~$j\to+\infty$. 
Let $x_0\in \fb u$ and set $u_j(x):=\frac{u(x_0+r_jx)}{r_j}$ such that $u_j\to u_0$
as~$j\to+\infty$
for some subsequence, still denoted by~$\{r_j\}$. 
If $n=2$ and $u^-$ is nondegenerate at $x_0$ then 
\[\lim_{r_j\to 0}\phi(r_j, x_0, u)=\gamma>0\]
and $u_0$ is a two-plane solution, namely
$$u_0(x)=\mu_1(x\cdot \ell)^+-\mu_2(x\cdot \ell)^-$$
for some unit direction  $\ell$ and positive constants $\mu_1, \mu_2 
$.
\end{lem} 

\begin{proof}
{F}rom the scale invariance of the Alt-Caffarelli-Friedman
functional we have that 
\[\phi(r_js, x_0, u)=\phi(s, 0, u_j).\]
Since, by Lemma~\ref{upng}, 
we have that~$u^+$ is nondegenerate, and by assumption so is $u^-$
at $x_0$, then it follows that 
the limit $$\lim_{r_j\to 0}\phi(sr_j, x_0, u)$$
exists and is independent of $s>0$, because $\phi$ is monotone and bounded thanks to Lipschitz continuity of $u$.
Therefore we have that \[\phi(s, 0, u_0)=\gamma>0, \qquad\forall s>0,\]
which implies that $u_0$ must be a
homogeneous function of degree 
1 (by the nondegeneracy of $u_0^+$ and the Lipschitz regularity of $u$,
recall Theorem~\ref{prop-nondeg} and Corollary~\ref{COR:7}).
Applying Lemma~6.6 in~\cite{ACF}, the desired result follows.
\end{proof}

Summarizing Proposition~\ref{NEWPF}
and Lemma~\ref{lem-nax-verjin}, we obtain
the result in Theorem~\ref{thm:reg}:
 
\begin{proof}[Proof of Theorem~\ref{thm:reg}]
By Lemma \ref{lem-visc-sol-full} we know that~$u$ is a viscosity solution.
It follows from Proposition~\ref{NEWPF}
and Lemma~\ref{lem-nax-verjin}
that the free boundary $\fb u$ is flat at each point. Hence, the proof of the theorem follows from the regularity theory of
Caffarelli developed for the viscosity solutions \cite{C-H1, C-H2}.
See also Proposition 6.1 in~\cite{DK}.
\end{proof}

\section*{Conflict of interests statement}{The authors declare that they do not have any conflict of interests.}

\vfill
\end{document}